\documentclass[11pt,a4paper,authoryear]{elsarticle}

\usepackage{wrapfig} 
\usepackage{flushend} 
\usepackage[english]{babel}
\usepackage[letterpaper,top=2cm,bottom=2cm,left=3cm,right=3cm,marginparwidth=1.75cm]{geometry}
\usepackage{url}
\usepackage[modulo]{lineno}
\usepackage{amsthm,amsmath,amssymb}
\usepackage{algorithm}
\usepackage{bm}
\usepackage{algpseudocode}
\usepackage{mathrsfs}
\usepackage[mathscr]{eucal}
\usepackage{exscale}
\usepackage[symbol,perpage]{footmisc}
\usepackage[colorlinks=true, allcolors=blue]{hyperref}
\setcitestyle{authoryear,round}
\usepackage{soul, color, xcolor}
\usepackage{textcomp}
\usepackage[title]{appendix}
\usepackage{tabularx}
\usepackage{longtable}
\usepackage{booktabs}
\usepackage{array}

\newtheorem{definition}{Definition}
\newtheorem{proposition}{Proposition}
\newtheorem{property}{Property}
\allowdisplaybreaks[3]
\begin{document}
	\begin{frontmatter}
		\title{Designing optimal subsidy schemes and recycling plans for sustainable treatment of construction and demolition waste}
		\date{}

		\author[a]{Lei Yu}
		\author[a]{Qian Ge\footnote{Corresponding author. Email: geqian@swjtu.edu.cn}}
		\author[b]{Ke Han\footnote{Corresponding author. Email: kehan@swjtu.edu.cn}}
		\author[a]{Wen Ji}
		\author[a]{Yueqi Liu}
		
		\affiliation[a]{organization={School of Transportation and Logistics, Southwest Jiaotong University},
			addressline={Xi'an 999, Pidu District}, 
			city={Chengdu},
			postcode={611756}, 
			state={Sichuan},
			country={China}}
		\affiliation[b]{organization={School of Economics and Management, Southwest Jiaotong University},
			addressline={No. 111, North Section 1, 2nd Ring Road}, 
			city={Chengdu},
			postcode={610031}, 
			state={Sichuan},
			country={China}}
		\begin{abstract}
			More than 10 billion tons of construction and demolition waste (CW) are generated globally each year, exerting a significant impact on the environment. In the CW recycling process, the government and the carrier are the two primary stakeholders. The carrier is responsible for transporting CW from production sites to backfill sites or processing facilities, with a primary focus on transport efficiency and revenue. Meanwhile, the government aims to minimize pollution from the recycling system, which is influenced by transport modes, shipment distances, and the processing methods used for CW. This paper develops a bi-objective, bi-level optimization model to address these challenges. The upper-level model is a linear programming model that optimizes the government's subsidy scheme, while the lower-level model is a minimum-cost flow model that optimizes the carrier's recycling plan. A hybrid heuristic solution method is proposed to tackle the problem's complexity. A case study in Chengdu, China, demonstrates the computational efficiency of the model and its small solution gap. With an optimized subsidy scheme and recycling plan, pollution can be reduced by over 29.29\% through a relatively small investment in subsidies.
		\end{abstract}
		
		\begin{keyword}
			Construction and demolition waste; Sustainable treatment; Minimum cost flow problem; Bi-level programming; Hybrid algorithm
		\end{keyword}
	\end{frontmatter}
	
	\section{Introduction}
	The global urbanization and urban renewal are leading to a significant increase in construction projects. With this comes a large amount of construction and demolition waste (CW), more than 10 billion tons of CW are generated globally annually~\citep{yazdani2021improving}, making it a major component of municipal solid waste. It accounts for 25-40\% in developing and developed countries~\citep{lin2020towards}. As the largest contributor, China generates more than 3 billion tons of CW annually, the United States generates more than 600 million tons annually, and the European Union generated 372 million tons in 2018~\citep{huang2018construction,zheng2024optimal}. Despite the enormous number, the recycling rate of the CW is quite small compared with other types of urban waste. In China, only 5\% of CW are recycled~\citep{huang2018construction}. Nonetheless, improper treatment can cause serious environmental hazards such as soil contamination and degradation, air and water pollution, greenhouse gas emissions, global warming, and severe health problems~\citep{rathore2020economic}. This highlights the need for an efficient, environmentally friendly CW recycling plan.
	
	Three types of sites are integral to the recycling process of CW: production sites, backfill sites, and processing facilities~\citep{chu2012optimization,yu2024using,yazdani2021improving}. The contractor generates waste at production sites and then commissions carrier to transport it and pays accordingly. The carrier's recycling plan significantly impacts the environment, influenced by transport modes and recycling methods. Diesel trucks are cheaper but more polluting, while electric trucks are cleaner but more costly, making carriers favor diesel trucks without government subsidies. Recycling methods also vary in environmental impact; some waste can be sent directly to backfill sites, while others require harmless treatment at facilities before recycling~\citep{chu2012optimization}. The processing facilities differ in pollutant outputs due to technological limitations. Carriers often prioritize convenience by selecting the nearest facilities, regardless of environmental concerns. 
	
	The core of the recycling plan lies in scheduling the movement of trucks among three types of sites. To maximize profit, carriers must decide on transport modes, recycling methods, and truck routes from production sites to backfill or processing facilities. However, these plans are often created manually based on experience, leading to two major drawbacks: (1) Due to the complexity of CW recycling and numerous influencing factors, manual scheduling often results in feasible but inefficient plans~\citep{chu2012optimization}; (2) The government’s goal of minimizing environmental pollution conflicts with the carrier’s objective of maximizing profit, leading to schedules that may not align with environmental priorities. Although the government cannot directly control carriers' recycling plans, it can indirectly influence them by adjusting subsidies and treatment fees at processing facilities to promote sustainability~\citep{chen2024construction}. The government's subsidy scheme interact with the carrier's recycling plan, as the two decisions are closely interdependent. The need to design an optimal subsidy scheme further complicates the problem.
	
	This study aims to address these challenges by developing an efficient and environmentally friendly scheduling method that incorporates the government's indirect influence through subsidies. Penalties are not considered, as they can lead to irregularities, potentially worsening safety and pollution issues~\citep{belien2014municipal}.
	
	The main contributions of this paper are as follows.
	
	\begin{itemize}
		\item[1)]  We formulate a tailored multi-vehicle minimum-cost flow model on a time-space network to optimize the carrier's recycling plan and maximize profit. The model improves solution efficiency by representing site congestion through service arcs, effectively modeling queues. Results from a large instance show that the model can be solved by commercial solver within a short time.
		
		\item[2)]  We formulate a bi-level optimization model to reduce pollution through a subsidy scheme, with the government as the leader and the carrier as the follower. The lower-level model is a multi-vehicle minimum-cost flow model, while the upper-level model minimizes environmental pollution by optimizing treatment fees across different facilities and truck types. 
		
		\item[3)] We develop an effective hybrid heuristic algorithm to solve the complex bi-level optimization model. The algorithm employs multi-objective particle swarm optimization (MOPSO) to explore solutions for the upper-level problem, while iteratively solving the lower-level problem to refine the best local optimum. It achieves a high-quality solution with a 1.51\% gap in a reasonable time of 3.76 hours.
	\end{itemize}
	
	\section{Literature review}\label{Sec:Lr}
	
	Dealing with CW is a global problem that requires a concerted effort from multiple entities. Research can be broadly categorized into technological and managerial methods~\citep{chen2024construction}. New technologies include introducing electrical trucks and environmentally friendly construction materials~\citep{pena2024sustainable}, etc. Section~\ref{Sec:lr_recycling} reviews the CW recycling process from a management perspective. Section~\ref{Sec:lr_subsidy} reviews studies that impose economic levers on CW management and presents research gaps.
	
	\subsection{CW management scheme}\label{Sec:lr_recycling}
	
	The framework of CW management can be divided into four stages: pre-construction, construction, transportation, and disposal\citep{galvez2018construction}. At the pre-construction stage, carriers need to prepare plans under the supervision of the government, such as estimating waste production, evaluating environmental impacts, and developing operational plans as well as economic drivers such as subsidies, fines, and taxes. During construction, the main considerations are reducing waste, reusing materials, storage, and sorting. The transportation process is quite important yet difficult to schedule. This stage is concerned with transportation efficiency, traffic safety, environmental pollution, etc. The disposal stage is mainly responsible for the harmless treatment of CW. The authorities set strict regulations to avoid illegal disposal~\citep{chen2024construction}.
	
	CW production estimation is widely studied in construction waste management and can be categorized into three methods: site visits (SV), generation rate calculation (GRC), and classification system accumulation (CSA)~\citep{lu2017estimating}. The choice of method depends on the specific objectives and conditions. SV involves field surveys, which are realistic but costly and difficult to replicate. GRC, the most common method, estimates CW based on waste generation rates for specific activities or companies. It is simple and cost-effective, but less accurate. CSA is more detailed, quantifying different types of CW, and offers more reliable data at a lower cost. Studies like \cite{solis2009spanish} and \cite{llatas2011model} applied CSA for more precise estimation, though some models are region-specific. \cite{guerra20204d} combined 4D-BIM for estimating concrete and waste. CSA provides effective, low-cost information for various CW types, forming a solid foundation for optimizing construction processes and waste management.
	
	Transportation of construction and demolition waste is one of the most important and expensive processes in groundwork due to the high volume of waste, time-sensitive production rate and processing cost, and demand for heavy trucks~\citep{aringhieri2018special}. Heavy trucks will affect the traffic operation of ordinary vehicles, inducing traffic congestion as well as serious accidents~\citep{esenduran2020choice}. They also account for a large proportion of emissions which therefore poses a great challenge to the carbon neutral and pollution reduction efforts~\citep{wijnsma2023treat}. The complexity of the transportation networks and designs can lead to inefficient transportation. 
	
	Optimization models for waste transportation have been extensively studied in the last decades. Although few studies focus particularly on the transport of construction waste, related waste transport studies are informative.~\cite{shi2020bi} formulated a bi-objective multi-period 0–1 integer programming model for household e-waste facility location, balancing cost and coverage during network expansion, and proposed tailored metaheuristics—especially a local-search-based approach—that demonstrate superior Pareto performance in extensive computational experiments.~\cite{kala2023note} studied solid waste collection and transfer using India as an example. They proposed two distinct vehicle routing algorithms: a simple and fast nearest-neighbor algorithm and a mixed-integer linear programming (MILP), both could reduce the total transportation cost significantly.

	~\cite{yazdani2021improving} proposed a similar heuristic algorithm based on integrated simulation optimization. Travel time was used as an objective function to optimize the route of trucks from construction sites to recycling facilities.~\cite{tirkolaee2020robust} proposed a MILP based on robust optimization techniques with total cost minimization as the objective function. The model was validated for different problems based on real data under deterministic and uncertainty conditions. There are also studies focusing on the pollution cost parameters and budget constraints. \cite{chu2012optimization} built a multi-commodity network flow model for the movements of trucks on the time-space network to optimize the system cost, and designed a heuristic method to solve it. They then performed a sensitivity analysis of fleet size, transportation cost, and scenarios to obtain decision-making recommendations. Most studies have focused on operational efficiency and ignored environmental issues.
	
	\subsection{Economic leverage for CW management scheme}\label{Sec:lr_subsidy}
	
	Economic leverage is a key method to influence carriers in waste transportation~\citep{esenduran2020choice, wiese2002waste}. The EU introduced waste disposal charges in 1999, with fees for inert materials like concrete and tiles set at 53 euros per cubic meter, while hazardous chemicals cost 86 euros~\citep{li2018willingness}. In Hong Kong, a waste disposal charge scheme has been in place since 2006. A study by \cite{hao2008effectiveness} found that the scheme effectively reduced waste production. Research on carriers' willingness to pay for treatment fees shows that while it exceeds current rates, it is still lower than the government's expectations\citep{li2020stakeholders}. ~\cite{chen2024construction} suggested that many charging policies focus more on improving transport efficiency than reducing pollution, and that Hong Kong's program may have limited long-term impact. ~\cite{elshaboury2022construction} noted that policies in China to reduce waste and promote recycling have had little effect.
	
	Compared with evaluation studies of construction waste charging programs, program design has been less studied in academia. In the case of normal operations, the government usually adopts incentives rather than penalties, as penalties can breed violations and lead to more serious pollution and safety problems~\citep{belien2014municipal}. Several studies have examined the impact of subsidies on increasing the use of new energy vehicles and thus reducing pollution~\citep{cheng2022cost}.~\cite{yuan2014system} pointed out that most regions in China have adopted treatment fees based on experience, with very limited effects. This study is the first to use a system dynamics approach to simulate the basic operation and policy analysis to determine an appropriate disposal fee scheme.~\cite{jia2017dynamic} investigated the effects of charging mechanisms on illegal dumping, recycling, and reuse. A system dynamics approach was used to determine a reasonable range of fees.~\cite{wijnsma2023treat} analyzed the effects of domestic and international waste regulations targeting dumping and export, respectively, on firm incentives and compliance. They established a two-tier waste chain, with producers responsible for generating waste and operators responsible for disposing it.~\cite{zheng2024optimal} developed differential game models to analyze the impact of government subsidies and consumer reference green effects on enterprise decisions and recycling rates in the CW market, offering insights for improving recycling efficiency and regulatory strategies.~\cite{xiong2017modelling} developed a game-theoretic model to examine pricing strategies and interactions among self-interested operators in a symbiotic waste management system with material exchanges. Numerical results highlighted that operating cost subsidies are more effective than gate fee controls in promoting new treatment technologies.~\cite{chen2024construction} proposed a methodology for the design of charges for general construction and demolition waste disposal that takes into account the behavior of carriers and the impact of waste transportation. A mixed integer programming model for the design of charging costs was developed to partially modify the current charging scheme to better achieve environmental protection objectives.
	
	In conclusion, the studies mentioned above have different focuses and significantly advance waste reduction and recycling efforts. However, they present several limitations: (1) Few studies simultaneously address the coordination of multiple trucks, multiple sites, and site service capacities, leading to suboptimal solutions; (2) Existing research often overlooks greenhouse gas emissions during transportation, which is significant in practice; (3) Most studies rely on system dynamics approaches to evaluate and determine treatment fee schemes, making it difficult to accurately model interactions between entities in the simulation system. In contrast, bi-level programming can explicitly capture the satisfactory strategies of both the local government and carriers, enabling them to make suitable decisions respectively at equilibrium. To bridge these gaps, we first construct a multi-vehicle minimum-cost flow model based on a time-space network to minimize carrier costs. Subsequently, we develop a bi-level optimization model based on the Stackelberg game framework to optimize treatment fee schemes and provide insights for better addressing environmental protection challenges.
	
	The following texts of this paper are organized as follows. Section~\ref{Sec:Problem_f} introduces the movements of the trucks among different types of sites and formulates the mathematical models. Section~\ref{Sec:Ma} designs a solution algorithm and evaluation metrics. The case study and conclusions are presented in Sections~\ref{Sec:Cs} and~\ref{Sec:Conclusions}, respectively.
	
	\section{Mathematical formulations}\label{Sec:Problem_f}
	
	This section begins by outlining the assumptions and setting of the research problem. Next, the time-space network for truck movements is introduced, which serves as the foundation for the carrier’s problem, i.e., the lower-level model. Finally, the upper-level problem, formulated as a bi-objective model, aims to optimize the government’s strategy by minimizing both pollution and subsidy expenditure. The complete bi-level problem is presented at the end.
	
	\subsection{Problem statement}\label{Sec:time-space}
	
	We develop a bi-objective, bi-level optimization model inspired by the Stackelberg game, where the government acts as the leader and the carrier as the follower. The government aims to minimize pollution and subsidy expenditure, while the carrier focuses on maximizing profit. Since pollution levels are influenced by truck scheduling, the government guides scheduling decisions by adjusting treatment fees. The structure of the bi-level model is illustrated in Fig.~\ref{Fig:structure_bilevel}.
	
	\begin{figure}[!h]
		\centering
		\includegraphics[width=0.8\textwidth]{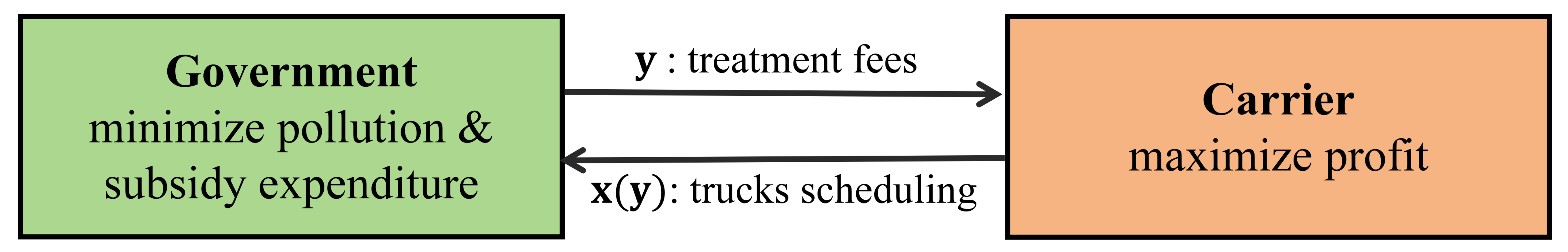}
		\caption{The structure of bi-level model.}\label{Fig:structure_bilevel}
	\end{figure}
	
	We consider both the $E$ electrical fleets $\mathcal{F}_e=\{\mathcal{V}_1,...,\mathcal{V}_E\}$ and $\overline{D}$ diesel fleets $\mathcal{F}_d=\{\mathcal{V}_{E+1},...,\mathcal{V}_{E+\overline{D}}\}$, where the number of truck in each fleet $v\in\mathcal{F}_e\cup\mathcal{F}_d$ is represented $N_v$. 
	
	Due to policy constraints, each truck is operated by a dedicated driver and runs for only one planning period before taking a mandatory rest. Since the battery range of electric trucks exceeds a single planning period, recharging is not a concern. Moreover, as trucks operate in a full-load mode between designated loading and unloading sites, the problem differs substantially from traditional vehicle routing problems in both structure and characteristics.
	
	Consequently, we develop the truck movements across different sites and periods as flows in a time-space network $\mathcal{G}=\{\mathcal{N},\mathcal{A}\}$, where $\mathcal N$ is a set of nodes and $\mathcal{A}$ is a set of directed arcs on the network as shown in Fig.~\ref{Fig:time_space}. Node set includes three types within the system: processing facilities $\mathcal{P}=\{1,..., P\}$, production sites $\mathcal{S}=\{P+1,..., P+S\}$, and backfill sites $\mathcal{D}=\{P+S+1,..., P+S+D\}$, with the total number is $S$, $P$, and $D$ respectively. We consider one depot, represented as the set $\{0\}$. The planning period $\mathcal{T}=\{0,1,...,T\}$, where $T$ represents the number of time intervals, $\Delta t$ represents time interval. To facilitate the mathematical formulation of the model, we define the virtual planning period $\overline{\mathcal{T}}=\{-1,...,-\overline{T}\}$, where $\overline{T}$ is a given integer. Each node $n_{i,t}\in\mathcal{N}$ represents site $i\in\{0\}\cup\mathcal{P}\cup\mathcal{S}\cup\mathcal{D}$ in period $t\in\mathcal{T}\cup\overline{\mathcal{T}}$. We assume that truck take $r_{i,j}$ time intervals to travel from area $i$ to area $j$ for $i,j\in\{0\}\cup\mathcal{P}\cup\mathcal{S}\cup\mathcal{D}$. The flow on the directed arc $(n_{i,t-r_{i,j}},n_{j,t})\in\mathcal{A}$ represents the number of truck moving from node $n_{i,t-r_{i,j}}$ to node $n_{j,t}$ for $i,j\in\{0\}\cup\mathcal{P}\cup\mathcal{S}\cup\mathcal{D}$ and $t\in\mathcal{T}\cup\overline{\mathcal{T}}$. Based on the real-world situation, we categorize the truck movements between two nodes into three categories, i.e., \textit{fully loaded arcs}, \textit{deadheading arcs}, and \textit{service arcs}. They are defined as follows.
	
	\begin{figure}[!t]
		\centering
		\includegraphics[width=0.8\textwidth]{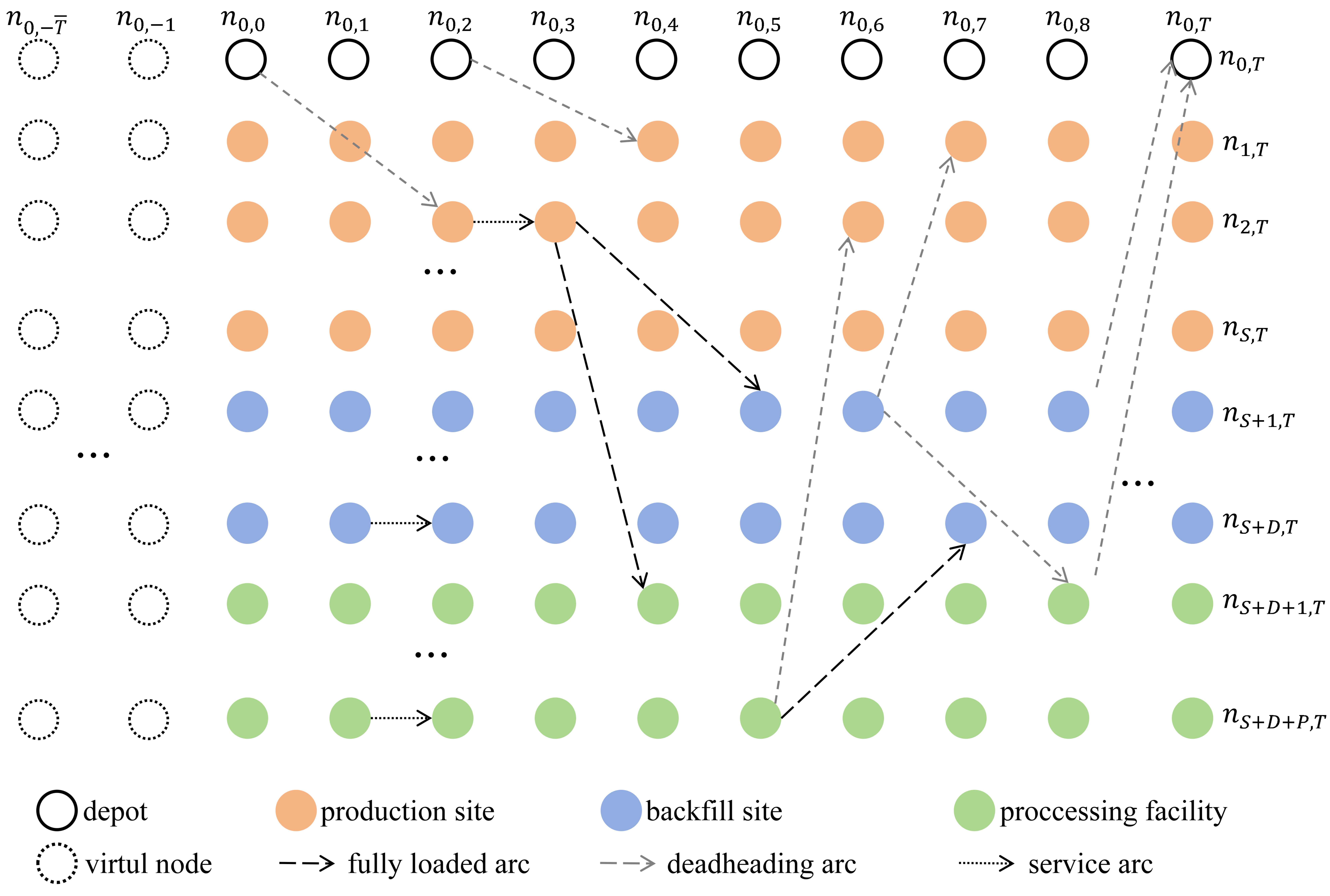}
		\caption{\label{Fig:time_space}Time-space network $\mathcal{G}$.}
	\end{figure}
	
	\begin{definition}[Fully loaded arc]
		A fully loaded arc $(n_{i,t-r_{i,j}},n_{j,t})\in\mathcal{A}^f$ represents the trip that trucks loaded with waste at node $n_{i,t-r_{i,j}}$ move to node $n_{j,t}$ to get unloaded, $t\in\mathcal{T}$, $i,j\in\{0\}\cup\mathcal{P}\cup\mathcal{S}\cup\mathcal{D}$.
	\end{definition}    
	Specifically, fully loaded arcs include three categories depending on the conditions $i, j$: 1) $i\in\mathcal{S}, j\in\mathcal{D}$, trucks transport waste from production sites to backfill sites; 2) $i\in\mathcal{S}, j\in\mathcal{P}$, trucks transport waste from production sites to processing facilities for harmless treatment or storage, etc; 3) $i\in\mathcal{P}, j\in\mathcal{D}$, when waste generated from production sites is not sufficient to meet the needs of backfill sites, trucks will transport waste from processing facilities to backfill sites for pit filling.
	
	\begin{definition}[Deadheading arc]
		A deadheading arc $(n_{i,t-r_{i,j}},n_{j,t})\in\mathcal{A}^d$ represents the trip in which trucks travel from node $n_{i,t-r_{i,j}}$ to node $n_{j,t}$ in empty, $t\in\mathcal{T}$, $i,j\in\{0\}\cup\mathcal{P}\cup\mathcal{S}\cup\mathcal{D}$.
	\end{definition}    
	Specifically, deadheading arcs include four categories depending on the conditions $i, j$: 1) $i\in\{0\}, j\in\mathcal{P}\cup\mathcal{S}$, at the beginning of the planning period, the trucks travel from the depot to production sites or processing facilities to load waste; 2) $i\in\mathcal{P}\cup\mathcal{D}, j\in\{0\}$, at the end of the planning period, trucks travel from backfill sites or processing facilities to the depot for rest; 3) $i\in\mathcal{D}, j\in\mathcal{P}\cup\mathcal{S}$, trucks travel to production sites or processing facilities to load waste after unloading at backfill sites; 4) $i\in\mathcal{P}, j\in\mathcal{S}$, similar to above, trucks travel to production sites to load waste after unloading at processing facilities.
	
	\begin{definition}[Service arc]
		A service arc $(n_{i,t_1},n_{i,t_2})\in\mathcal{A}^s$ represents the waiting phase in which trucks arrive at node $n_{i,t_1}$ and then moves to node $n_{i,t_2}$, $i\in\{0\}\cup\mathcal{P}\cup\mathcal{S}\cup\mathcal{D}$, $t_1,t_2\in\mathcal{T}$. In other words, the trucks stay at site $i$ for one time intervals $\Delta t$.
	\end{definition}
	
	Upon arrival, trucks queue, load (or unload), check, clean, and then depart. The total service time depends on the site’s capacity. For simplicity, we assume the service time is a time interval $\Delta t$, which is typically longer than the actual service time, providing a buffer for travel and delays. Similar assumptions are used in transportation planning~\citep{carey2003pseudo}. Thus, the service arc $(n_{i,t-1}, n_{i,t}) \in \mathcal{A}^s$ represents the waiting phase where trucks arrive at node $n_{i,t-1}$, stay for $\Delta t$, and then depart from node $n_{i,t}$, with $i \in \mathcal{P} \cup \mathcal{S} \cup \mathcal{D}$ and $t \in \mathcal{T}$.
	
	Except for the three types of arcs mentioned, all other arcs in $\mathcal{A}$ have a flow of 0, represented as the set $\mathcal{A}^0$. We define $\mathcal{I} = \{f, d, s, 0\}$, so that $\mathcal{A} = \cup_{\forall e \in \mathcal{I}} \mathcal{A}^e$ and $\mathcal{A}^{e_1} \cap \mathcal{A}^{e_2} = \emptyset$ for $e_1 \neq e_2, e_1, e_2 \in \mathcal{I}$. We then formulate an integer programming problem based on the time-space network in Fig.~\ref{Fig:time_space} to maximize the carrier's profit.
	
	\subsection{The strategy of the carrier}\label{Sec:mathematical_m}
	
	In this section, we formulate a model based on the time-space network $\mathcal{G}$ to maximize the carrier's profit (or equivalently, minimize the negative of the carrier's profit), with three components: 1) Fixed cost, where $C_{0,v}$ is the fixed cost per truck in fleet $v \in \mathcal{F}_e \cup \mathcal{F}_d$, covering the driver's salary, truck maintenance, etc.~\citep{belien2014municipal}. 2) Travel cost, where $C_{1,v}$ is the average cost per truck in fleet $v$ for traveling one time interval $\Delta t$, mainly covering fuel or electricity costs. 3) Revenue from waste transport, with a unit price of $C_2$ CNY per tonne. However, if the waste is sent to processing facilities, the carrier incurs an additional treatment fee, which is either a fixed market price $y^{\prime}$ or a variable fee, discussed in Section~\ref{Sec:Mathematical_BI}. The task volume for each planning period is known in advance, and CW accumulation costs on-site are excluded.
	
	We define decision variable $x_{i,j,v,t}$ as the flow of fleet $v\in\mathcal{F}_e\cup\mathcal{F}_d$ from node $n_{i,t}\in\mathcal{N}$ to another node $n_{j,t+r_{i,j}}\in\mathcal{N}$ at time $t\in\mathcal{T}\cup\overline{\mathcal{T}}$. All of the notations are shown in~\ref{Appendix:Notation}, and the mathematical model is as follows:
	
	\begin{align}
		&\text{\bf{[M1]}}\notag\\ &\text{Min }
		\begin{aligned}[t]
			f(\mathbf{x})=& \overbrace{\sum_{v\in\mathcal{F}_e\cup\mathcal{F}_d}\sum_{j\in\mathcal{P}\cup\mathcal{S}\cup \mathcal{D}}\sum_{t\in\mathcal{T}}C_{0,v}x_{0,j,v,t}}^\text{fixed cost} \\
			& +\overbrace{\sum_{v\in\mathcal{F}_e\cup\mathcal{F}_d}\sum_{i\in\mathcal{P}\cup\mathcal{S}\cup\mathcal{D}}\sum_{j\in\mathcal{P}\cup\mathcal{S}\cup\mathcal{D}}\sum_{t\in\mathcal{T}}C_{1,v}r_{i,j}x_{i,j,v,t}}^\text{travel cost} \\
			& -\overbrace{(\sum_{v\in\mathcal{F}_e\cup\mathcal{F}_d}\sum_{i\in S\cup\mathcal{P}}\sum_{j\in\mathcal{D}}\sum_{t\in\mathcal{T}}C_2Q_vx_{i,j,v,t}+\sum_{v\in\mathcal{F}_e\cup\mathcal{F}_d}\sum_{i\in\mathcal{S}}\sum_{j\in\mathcal{P}}\sum_{t\in\mathcal{T}}(C_2-y^{\prime})Q_vx_{i,j,v,t})/1,000}^\text{revenue from transporting CW};
		\end{aligned}\label{M1:obj} \\
		&\text{Subject to:} \notag \\
		& \qquad \sum_{j\in\mathcal{P}\cup\mathcal{S}\cup\mathcal{D}}\sum_{t\in\mathcal{T}}x_{0,j,v,t}\leq N_v,\forall v\in\mathcal{F}_e\cup\mathcal{F}_d; \label{M1:start}\\
		& \qquad \sum_{j\in\mathcal{P}\cup\mathcal{S}\cup\mathcal{D}}\sum_{t\in\mathcal{T}}x_{0,j,v,t}-\sum_{i\in\mathcal{P}\cup\mathcal{S}\cup\mathcal{D}}\sum_{t\in\mathcal{T}}x_{i,0,v,t}=0,\forall v\in\mathcal{F}_e\cup\mathcal{F}_d; \label{M1:end}\\
		& \qquad \sum_{i\in\{0\}\cup \mathcal{P}\cup \mathcal{S}\cup \mathcal{D}}x_{i,j,v,t-r_{i,j}-1}-\sum_{i\in\{0\}\cup \mathcal{P}\cup \mathcal{S}\cup \mathcal{D}}x_{j,i,v,t}=0,\forall j\in\mathcal{P}\cup\mathcal{S}\cup\mathcal{D},\forall v\in\mathcal{F}_e\cup\mathcal{F}_d,\forall t\in\mathcal{T}; \label{M1:flow_E}\\
		& \qquad \sum_{i\in\{0\}\cup\mathcal{P}\cup\mathcal{S}\cup\mathcal{D}}\sum_{v\in\mathcal{F}_e\cup\mathcal{F}_d}x_{i,j,v,t-r_{i,j}}\leq B_j,\forall j\in\mathcal{P}\cup\mathcal{S}\cup\mathcal{D},\forall t\in\mathcal{T}; \label{M1:restrict_B}\\
		& \qquad \sum_{j\in\mathcal{P}\cup\mathcal{D}}\sum_{v\in\mathcal{F}_e\cup\mathcal{F}_d}\sum_{t\in\mathcal{T}}Q_vx_{i,j,v,t}/1,000\geq qs_i,\forall i\in\mathcal{S}; \label{M1:qs_L}\\
		& \qquad \sum_{j\in\mathcal{P}\cup\mathcal{D}}\sum_{v\in\mathcal{F}_e\cup\mathcal{F}_d}\sum_{t\in\mathcal{T}}Q_vx_{i,j,v,t}/1,000< (1+\epsilon_2)qs_i,\forall i\in\mathcal{S}; \label{M1:qs_U}\\
		& \qquad \sum_{i\in\mathcal{P}\cup\mathcal{S}}\sum_{v\in\mathcal{F}_e\cup\mathcal{F}_d}\sum_{t\in\mathcal{T}}Q_vx_{i,j,v,t}/1,000\geq qd_j,\forall j\in\mathcal{D}; \label{M1:qd_L}\\
		& \qquad \sum_{i\in\mathcal{P}\cup\mathcal{S}}\sum_{v\in\mathcal{F}_e\cup\mathcal{F}_d}\sum_{t\in\mathcal{T}}Q_vx_{i,j,v,t}/1,000< (1+\epsilon_2)qd_j,\forall j\in\mathcal{D}; \label{M1:qd_U}\\
		& \qquad x_{i,j,v,t}=0,\forall i\in\{0\}\cup\mathcal{P}\cup\mathcal{S}\cup\mathcal{D},\forall j\in\{0\}\cup\mathcal{P}\cup\mathcal{S}\cup\mathcal{D},\forall v\in\mathcal{F}_e\cup\mathcal{F}_d,\forall t\in\overline{\mathcal{T}}; \label{M1:zero1}\\
		& \qquad x_{i,j,v,t}=0, \exists\mathcal{H}\in\{\{0\},\mathcal{P},\mathcal{S},\mathcal{D}\},\forall i,j\in \mathcal{H},\forall v\in\mathcal{F}_e\cup\mathcal{F}_d,\forall t\in\mathcal{T}; \label{M1:zero2}\\
		& \qquad x_{i,j,v,t}\in\{0,1,...,B_j\},\forall i,j\in\{0\}\cup\mathcal{P}\cup S\cup D,\forall v\in\mathcal{F}_e\cup\mathcal{F}_d,\forall t\in \mathcal{T}\cup\overline{\mathcal{T}}.\label{M1:zero3}
	\end{align}
	
	Constraint~\eqref{M1:start} ensures that the total number of trucks traveling from the depot to the sites is less than the total number of trucks in the fleet. Constraint~\eqref{M1:end} ensures that the number of trucks traveling from the depot equals the number of trucks returning to the depot. Constraint~\eqref{M1:flow_E} indicates that the node flow is conserved. Constraint~\eqref{M1:restrict_B} restricts the maximum number of trucks arriving at a site simultaneously to avoid congestion. Constraints~\eqref{M1:qs_L} and~\eqref{M1:qs_U} ensure that the CW generated from each production site is transported, where $\epsilon_2$ is a smaller number. Constraints~\eqref{M1:qd_L} and~\eqref{M1:qd_U} ensure that the needs of each backfill site are met. Constraints~\eqref{M1:zero1} and~\eqref{M1:zero2} filter infeasible transport arcs. Constraint~\eqref{M1:zero3} is an integer constraint that the upper flow limit should be equal to the operating capacity of the operating area. Solving the model $\bf{[M1]}$ generates an efficient scheduling for the carrier.
	
	\subsection{The strategy of the government}\label{Sec:mathematical_BI}
	
	The profit objective for subsidizing is straightforward. We first demonstrate how the pollution related objective is computed in  Section~\ref{Sec:Objective_UL} and present the full bi-objective bi-level model in Section~\ref{Sec:Mathematical_BI}.
	
	\subsubsection{Calculate method of pollution}\label{Sec:Objective_UL}
	
	Pollution from CW transportation consists of two components: emissions from diesel trucks during transportation and pollution from waste treatment at processing facilities. To improve realism, truck emissions are typically estimated using the Comprehensive Modal Emissions Model (CMEM)~\citep{DEMIR2014464}. Based on this model, the Fuel Consumption Rate (FCR) is given by:
	\begin{equation}
		FCR=\frac{\xi(kNV+P_0/\eta)}{\kappa},
	\end{equation}
	where $\xi$ is the fuel-to-air mass ratio, $k$ is the engine friction factor, $N$ and $V$ denote the engine speed and engine displacement, respectively. The parameters $\eta$ and $\kappa$ are constants representing efficiency and the heating value, respectively. $P_0$ is the second-by-second engine power output (in kilowatts) and can be calculated as:
	\begin{equation}
		P_0=\frac{P_{t}}{\eta_{t}}+P_{a},
	\end{equation}
	where $\eta_{t}$ denotes the truck drive train efficiency, and $P_a$ is the engine power requirement associated with the operating losses of the engine and the operation of truck accessories (e.g., air conditioning). $P_a$ is typically assumed to be 0~\citep{DEMIR2011347}. $P_t$ represents the total traction force, which is the force generated by the friction between the truck tires and the road surface. $P_t$ can be calculated as follows:
	\begin{equation}
		P_t=\frac{(M_v\tau+M_vgsin\delta+0.5f_a\varphi Au+M_vgf_rcos\delta)u}{1,000},
	\end{equation}
	where $M_v$ is the truck's total weight in feet $v \in \mathcal{F}_e \cup \mathcal{F}_d$, including both the unloaded weight ($\overline{Q}_v$) and rated load ($Q_v$). $\tau$ is acceleration, $g$ is gravitational acceleration, and $\delta$ is the road angle. The coefficients $f_a$ and $f_r$ represent aerodynamic drag and rolling resistance, respectively. $\varphi$ is air density, and $A$ is the truck's frontal area. $u$ is the truck's speed. For simplicity, \cite{DEMIR2014464} defines $\lambda = \frac{\xi}{\kappa \theta}$, $\gamma = \frac{1}{1,000\eta_t\eta}$, $\alpha = \tau + g \sin \delta + g f_r \cos \delta$, and $\beta = 0.5 f_d \varphi A$, where $\theta$ is the fuel conversion factor from $\frac{gram}{s}$ to $\frac{liter}{s}$, and $\xi$ is typically assumed to be 1~\citep{DEMIR2011347}. Thus, the pollution index (Fuel Consumption, FC) for a diesel truck $v \in \mathcal{F}_d$ over a distance $d$ is a function of speed $u$ and total weight $M = \overline{Q}_v + Q_v$:
	
	\begin{equation}\label{Eq:FC}
		FC(u,M)=\frac{\lambda(kNV+\overline{Q}_v \gamma \alpha u+Q_v \gamma \alpha u+\beta\gamma u^3)d}{u},\forall v\in\mathcal{F}_d.
	\end{equation}

	Pollution from processing facilities primarily depends on the facility's capacity and the volume of CW processed. The pollution generated by the harmless treatment of one tonne of CW at a facility $j \in \mathcal{P}$ is defined by the pollution factor $h_j$. Therefore, the pollution index for processing a truckload of CW from fleet $v \in \mathcal{F}_e \cup \mathcal{F}_d$ at facility $j \in \mathcal{P}$ is expressed as:
	\begin{equation}\label{Eq:PTC}
		HTC=h_jQ_v,\forall j\in\mathcal{P},\forall v\in\mathcal{F}_e\cup\mathcal{F}_d.
	\end{equation}
	
	\subsubsection{Mathematical formulation}\label{Sec:Mathematical_BI}
	
	The objective of the upper-level problem is to minimize pollution induced by CW transport and processing while minimizing the subsidy from the government. The compound objective function could be derived from~\eqref{Eq:FC} and~\eqref{Eq:PTC}. The lower-level problem corresponds to $\bf{[M1]}$. The decision variables for the upper-level problem are $y_{j,v}$s, representing the treatment fee for the harmless treatment per ton of CW carried by fleet $v\in\mathcal{F}_e\cup\mathcal{F}_d$ at processing facility $j\in\mathcal{P}$. They form input parameters to the lower-level problem. The decision variable for the lower-level problem is $x_{i,j,v,t}$, and the optimal solution for the lower-level problem is denoted by $x^{\prime}_{i,j,v,t}$, $i,j\in\{0\}\cup\mathcal{P}\cup S\cup D, v\in\mathcal{F}_e\cup\mathcal{F}_d, t\in \mathcal{T}\cup\overline{\mathcal{T}}$. The bi-level program is given by:
	
	\begin{align}
		&\text{\bf{[M2]}}\notag\\ &\text{Min }
		\begin{aligned}[t]
			F_{1}(\mathbf{x}^{\prime}|\mathbf{y})=& \sum_{i\in\mathcal{P}\cup\mathcal{S}\cup\mathcal{D}}\sum_{j\in\mathcal{P}\cup\mathcal{S}\cup\mathcal{D}}\sum_{v\in\mathcal{F}_d}\sum_{t\in\mathcal{T}}kNV\lambda d_{i,j}x_{i,j,v,t}^{\prime}/u \\
			& +\sum_{i\in\mathcal{P}\cup\mathcal{S}\cup\mathcal{D}}\sum_{j\in\mathcal{P}\cup\mathcal{S}\cup\mathcal{D}}\sum_{v\in\mathcal{F}_d}\sum_{t\in\mathcal{T}}\overline{Q}_v\gamma\lambda\alpha_{i,j}d_{i,j}x_{i,j,v,t}^{\prime} \\
			& +(\sum_{i\in\mathcal{S}}\sum_{j\in\mathcal{P}\cup\mathcal{D}}\sum_{v\in\mathcal{F}_d}\sum_{t\in\mathcal{T}}Q_v\gamma\lambda\alpha_{i,j}d_{i,j}x_{i,j,v,t}^{\prime}+\sum_{i\in\mathcal{P}}\sum_{j\in\mathcal{D}}\sum_{v\in\mathcal{F}_d}\sum_{t\in\mathcal{T}}Q_v\gamma\lambda\alpha_{i,j}d_{i,j}x_{i,j,v,t}^{\prime}) \\
			& +\sum_{i\in\mathcal{P}\cup\mathcal{S}\cup\mathcal{D}}\sum_{j\in\mathcal{P}\cup\mathcal{S}\cup\mathcal{D}}\sum_{v\in\mathcal{F}_d}\sum_{t\in\mathcal{T}}\beta\gamma\lambda d_{i,j}x_{i,j,v,t}^{\prime}u^2 \\
			& +\sum_{i\in\mathcal{S}}\sum_{j\in\mathcal{P}}\sum_{v\in\mathcal{F}_e\cup\mathcal{F}_d}\sum_{t\in\mathcal{T}}h_jQ_vx_{i,j,v,t}^{\prime};
		\end{aligned}\label{M2:obj} \\
		&\text{Subject to:} \notag \\
		&\qquad y_{j,v}\geq PL,\forall j\in\mathcal{P},\forall v\in\mathcal{F}_e\cup\mathcal{F}_d; \label{M2:PL}\\
		&\qquad y_{j,v}\leq PU,\forall j\in\mathcal{P},\forall v\in\mathcal{F}_e\cup\mathcal{F}_d; \label{M2:PU}\\
		\text{where} \notag \\
		&\begin{aligned}
			\mathbf{x}^{\prime}\in \arg\min\{f^{\prime}(\mathbf{x},\mathbf{y})=&\sum_{v\in\mathcal{F}_{e}\cup\mathcal{F}_{d}}\sum_{j\in\mathcal{P}\cup\mathcal{S}\cup D}\sum_{t\in\mathcal{T}}C_{0,v}x_{0,j,v,t} \\
			&+\sum_{v\in\mathcal{F}_{e}\cup\mathcal{F}_{d}}\sum_{i\in\mathcal{P}\cup\mathcal{S}\cup\mathcal{D}}\sum_{j\in\mathcal{P}\cup\mathcal{S}\cup\mathcal{D}}\sum_{t\in\mathcal{T}}C_{1,v}r_{i,j}x_{i,j,v,t} \\
			&-\sum_{v\in\mathcal{F}_{e}\cup\mathcal{F}_{d}}\sum_{i\in\mathcal{S}\cup\mathcal{P}}\sum_{j\in\mathcal{D}}\sum_{t\in\mathcal{T}}C_{2}Q_{v}x_{i,j,v,t}\\
			&-\sum_{v\in\mathcal{F}_{e}\cup\mathcal{F}_{d}}\sum_{i\in\mathcal{S}}\sum_{j\in\mathcal{P}}\sum_{t\in\mathcal{T}}(C_{2}-y_{j,v})Q_{v}x_{i,j,v,t}\};
		\end{aligned}\label{M2:LL_obj}\\
		&\text{Subject to:} \notag\\
		&\qquad \text{Constraints~\eqref{M1:start}-\eqref{M1:zero3}.} \notag
	\end{align}
	
	Constraint~\eqref{M2:PL} limits the treatment fee to the government’s budget, with $PL$ representing the minimum acceptable price. Constraint~\eqref{M2:PU} sets the maximum price, $PU$, acceptable to the carrier; exceeding this may lead to illegal shipment or processing risks~\citep{belien2014municipal}. Constraint~\eqref{M2:LL_obj} minimizes the negative of the carrier’s profit (i.e., maximize the carrier’s profit), derived from function~\eqref{M1:obj} by replacing $y^{\prime}$ with the variable $y_{j,v}, j\in\mathcal{P}, v\in\mathcal{F}_e\cup\mathcal{F}_d$. Model $\bf{[M2]}$ has the following property:
	
	\begin{property}\label{Property:p1}
		As subsidy $\sum_{v\in\mathcal{F}_{e}\cup\mathcal{F}_{d}}\sum_{i\in\mathcal{S}}\sum_{j\in\mathcal{P}}\sum_{t\in\mathcal{T}}y_{j,v}Q_{v}x_{i,j,v,t}$ increases, the objective of model $\bf{[M2]}$ is non-increasing. 
	\end{property}
	
	\begin{property}\label{Property:p2}
		The objective of $\bf{[M2]}$ has a natural lower bound, e.g., pollution will not be less than 0.
	\end{property}
	
	The proofs are omitted because they are straightforward. Property~\ref{Property:p1} shows that the government can reduce pollution by increasing the subsidy investment. However, Property~\ref{Property:p2} indicates that beyond a certain subsidy level, further increases do not result in additional pollution reduction, leading to multiple optimal solutions for $\bf{[M2]}$. Since the government aims to minimize pollution with the least subsidy investment, we introduce a second objective function to minimize the subsidy. When processing facility $j\in\mathcal{P}$ treats a load of waste from fleet $v\in\mathcal{F}_e\cup\mathcal{F}_d$, the carrier should pay the government treatment fee:
	\begin{equation}
		\mathrm{CWTF}=y_{j,v}Q_v,\forall j\in\mathcal{P},\forall v\in\mathcal{F}_e\cup\mathcal{F}_d.
	\end{equation}
	
	To simplify the mathematical formulation, let $\mathrm{CWTF}$ represent the government’s revenue from the harmless treatment of one tonne of waste, which is inversely proportional to the subsidy investment. Accordingly, the bi-level model is formulated as follows:
	
	\begin{align}
		&\text{\bf{[M3]}}\notag\\ &\text{Min }
		\begin{aligned}[t]
			F_{2}(\mathbf{x},\mathbf{y})=-\sum_{i\in \mathcal{S}}\sum_{j\in\mathcal{P}}\sum_{v\in\mathcal{F}_e\cup\mathcal{F}_d}\sum_{t\in\mathcal{T}}y_{j,v}Q_vx_{i,j,v,t}^{\prime};
		\end{aligned}\label{M3:obj} \\
		&\text{Subject to:} \notag \\
		&\qquad F_{1}(\mathbf{x}^{\prime}|\mathbf{y})=F_{1}(\mathbf{x}^*_{M2}); \label{M3:UP_obj}\\
		&\qquad \text{Constraints}~\eqref{M2:PL}-\eqref{M2:PU}; \notag \\
		\text{where} \notag \\
		&\begin{aligned}
			\mathbf{x}^{\prime}\in \arg\min\{f^{\prime}(\mathbf{x},\mathbf{y})\};
		\end{aligned}\label{M3:LL_obj}\\
		&\text{Subject to:} \notag\\
		&\qquad \text{Constraints~\eqref{M1:start}-\eqref{M1:zero3}.} \notag
	\end{align}
	
	The objective function~\eqref{M3:obj} minimizes the subsidy investment of the government. The function $F_{1}(\cdot)$ in constraint~\eqref{M3:UP_obj} is derived from objective function~\eqref{M2:obj}, where $\mathbf{x}^*_{M2}$ represents the optimal solution of $\mathbf{x}^*$ acquired by solving \textbf{[M2]}, and $F_{1}(\mathbf{x}^*_{M2})$ is optimal value for \textbf{[M2]}. Therefore, constraint~\eqref{M3:UP_obj} ensures that the total pollution must be minimized. Constraint~\eqref{M3:LL_obj} is identical to constraint~\eqref{M2:LL_obj}. Solving model \textbf{[M3]} yields the economical and environmentally friendly scheduling.
	
	\section{Solution approach}\label{Sec:Ma}
	
	Section~\ref{Sec:solution_a} outlines a hybrid approach to address the bi-objective bi-level optimization model \textbf{[M3]}. The method for calculating the model gap is detailed in Section~\ref{Sec:GAP}.
	
	\subsection{Solution algorithm}\label{Sec:solution_a}
	
	\subsubsection{Global framework}
	
	For the model \textbf{[M1]}, a commercial solver is used. For the model \textbf{[M3]}, the hybrid approach combines a multi-objective particle swarm optimization (MOPSO)~\citep{1304847} to search for the solution to the upper-level problem and a commercial solver to solve the lower-level problem. We first define the total number of fleets $V$ (i.e., $V=E+\overline{D}$). The hybrid approach begins by randomly initializing a particle swarm of $K$ particles $\mathbf{y}^k=(y_{1,1}^k,\ldots,y_{1, V}^k,y_{2,1}^k,\ldots,y_{j,v}^k,\ldots,y_{P, V}^k)$, $k=1,\ldots, K$, representing the treatment fee set by government at processing facility $j\in\mathcal{P}$ for truck fleet $v\in\mathcal{F}_e\cup\mathcal{F}_d$. Meanwhile, randomly initialize the corresponding velocity $\bm{\mu}^k=(\mu_{1,1}^k,\ldots,\mu_{1,V}^k,\mu_{2,1}^k,\ldots,\mu_{j,v}^k,\ldots,\mu_{P,V}^k)$ of each particle $\mathbf{y}^k$. We input each $\mathbf{y}^k$ into the lower-level problem of \textbf{[M3]} to obtain a solution $\mathbf{x}^k$, where $\mathbf{x}$ represents the vector of all decision variables of the lower-level problem. The pseudocode outlines the steps of the proposed hybrid approach as follows:
	
	\begin{algorithm}
		\caption{Pseudocode of the Hybrid Approach}
		\label{Algorithm-main}
		\begin{algorithmic}[1]
			\State Initialize a particle swarm of $K$ particles $\mathbf{y}^k=(y_{1,1}^k,\ldots,y_{1,V}^k,y_{2,1}^k,\ldots,y_{j,v}^k,\ldots,y_{P,V}^k)$, velocity $\bm{\mu}^k=(\mu_{1,1}^k,\ldots,\mu_{1,V}^k,\mu_{2,1}^k,\ldots,\mu_{j,v}^k,\ldots,\mu_{P,V}^k)$, $k = 1, \ldots, K$, $j\in\mathcal{P}$, $v\in\mathcal{F}_e\cup\mathcal{F}_d$
			\Repeat
			\For{$k = 1$ to $K$}
			\State Repair $\mathbf{y}^k$ to satisfy constraints~\eqref{M2:PL} and~\eqref{M2:PU}, retaining each $y_{j,v}^k$ to five decimal digits
			\State Input $\mathbf{y}^k$ into the lower level problem and solve it using solver to obtain $\mathbf{x}^k$
			\State Compute $F_{1}(\mathbf{x}^k)$ and $F_{2}(\mathbf{x}^k, \mathbf{y}^k)$ for each $(\mathbf{x}^k, \mathbf{y}^k)$
			\State Update $\mathbf{y}^k$ using MOPSO (introduced in Section~\ref{Sec:MOPSO})
			\EndFor
			\Until{$G$ iterations are completed}
			\State \text{Output:} $(\mathbf{x}^k, \mathbf{y}^k)$ with the best $F_{1}(\mathbf{x}^k)$
		\end{algorithmic}
	\end{algorithm}
	
	Line 4 of Algorithm~\ref{Algorithm-main} requires a repair routine that projects any $y_{j,v}^k \in \mathbf{y}^k$ to $PL$ or $PU$ if it violates constraint~\eqref{M2:PL} or~\eqref{M2:PU}. The velocity of the particle is then updated by applying a negative sign, prompting the particle to search in the opposite direction. Since the treatment fee must be a fixed value in practice, even though $y_{j,v}^k$ is treated as a continuous variable in the model, it is truncated to five decimal places after initialization or updating.
		
		\subsubsection{Upper-level search based on MOPSO}\label{Sec:MOPSO}
		
		The upper constraint~\eqref{M3:UP_obj} of model \textbf{[M3]} can create challenges in generating feasible solutions $\mathbf{y}^k$, $k=1,\ldots, K$, since it represents the objective of minimizing pollution in the upper-level problem. To address this, we treat functions $F_1$ and $F_2$ as the primary and secondary objective functions, respectively. Previous experiments have shown that heuristic methods designed for single-objective optimization often focus solely on the primary objective function $F_1$, largely ignoring $F_2$. To overcome this, we initially assign equal priority to both $F_1$ and $F_2$ and use MOPSO to identify the Pareto optimal solution set. Afterward, we select the solution with the minimum value of $F_1$ from this set as the final high-quality solution. This approach ensures that the solution converges toward the objectives of both $F_1$ and $F_2$.
		
		Line 6 of Algorithm~\ref{Algorithm-main} calculates $\mathbf{x}^k$, $\mathbf{y}^k$, $F_1(\mathbf{x}^k)$, and $F_2(\mathbf{x}^k, \mathbf{y}^k)$. Here, $\mathbf{y}^k$ is a particle in MOPSO, while $F_1(\mathbf{x}^k)$ and $F_2(\mathbf{x}^k, \mathbf{y}^k)$ represent the fitness values corresponding to the two objective functions for the current $\mathbf{x}^k$ and $\mathbf{y}^k$. MOPSO iteratively adjusts the particles to move toward optimal positions. The velocity $\bm{\mu}^k$ of each particle is influenced by both its personal best position ($\mathbf{ybest}^k$) and the position of the global best particle ($\mathbf{gbest}^k$). The global best selection mechanism follows the approach by~\cite{1304847}, with the main steps summarized as follows: (1) Define three arrays: $EAY$, $EAF_1$, and $EAF_2$, each of length $M$, to store the positions of Pareto optimal particles and their corresponding objective function values $F_1$ and $F_2$, respectively. (2) Divide the objective space into hypercubes by creating a grid with $HQ = m \times m$ cells. Each hypercube represents a region, and the particle density in each region is used to calculate its distribution. (3) Use a roulette-wheel mechanism to select a global best particle ($\mathbf{gbest}^k$) from $EAY$. Particles in less populated hypercubes are given higher selection probabilities, ensuring that under-explored regions are prioritized (e.g., selecting $EAY^q[i]$ at iteration $q$). (4) If the number of Pareto solutions exceeds the length of $EAY$, particles in densely populated hypercubes are more likely to be deleted, preserving diversity in the repository. For additional technical details, refer to~\cite{1304847}.
		
		In any iteration $q$, the velocity component $\mu_{j,v}^{k,q}$ of the particle $\mathbf{y}^k$ is updated according to the following equation~\citep{kennedy1995particle}:
		\begin{equation}
			\mu_{j,v}^{k,q}=\omega_{0}\mu_{j,v}^{k,q-1}+\omega_{1}R_{1}(pbest_{j,v}^{k}-y_{j,v}^{k,q-1})+\omega_{2}R_{2}(gbest_{j,v}-y_{j,v}^{k,q-1}),
		\end{equation}
		where $\omega_{0}$ represents the inertia weight, $\omega_{1}$ and $\omega_{2}$ represent cognitive and social perceptual parameters; $R_1$ and $R_2$ are random numbers uniformly distributed in the interval [0,1]; $gbest_{j,v}=EAY[i]_{j,v}, i\in\{1,\ldots, M\}$. The new position of particle $\mathbf{y}^k$ at iteration $q$ is calculated using the following equation~\citep{kennedy1995particle}:
		\begin{equation}\label{y_update}
			\mathbf{y}^{k,q}=\mathbf{y}^{k,q-1}+\bm{\mu}^{k,q}.
		\end{equation}
		
		Particle swarm optimization converges quickly and easily falls into local optimal solutions, so we introduce random perturbations into the algorithm based on study~\cite{soares2020designing}. If the solution does not change after $G^{'}$ iterations, the algorithm introduces perturbations with probability $p_m$, enhancing the exploration capability of the particles. Specifically, for each particle $\mathbf{y}^k$, after being updated by Eq.~\eqref{y_update} and before being repaired (Line 4 in Algorithm~\ref{Algorithm-main}), a perturbation $\varepsilon$ is applied within the range $\left[-\sigma\left(PU-PL\right),\sigma\left(PU-PL\right)\right],\mathrm{i.e.},y_{j,v}^{k,q}\leftarrow y_{j,v}^{k,q}+\varepsilon$. The algorithm terminates when the difference between objective function values from two consecutive iterations falls below a specified threshold $\epsilon_1$, In other words, the following two inequalities are satisfied:
		\begin{equation}
			\frac{EAF_1^q[index]-EAF_1^{q-1}[index]}{EAF_1^q[index]}<\epsilon_1,
		\end{equation}
		\begin{equation}
			\frac{EAF_2^q[index]-EAF_2^{q-1}[index]}{EAF_2^q[index]}<\epsilon_1,
		\end{equation}
		where $EAF_{1}^{q}[index] = \min\{EAF_{1}^{q}[1],\ldots,EAF_{1}^{q}[M]\}$, $EAF_2^q[index]$ is optimal under the condition that $EAF_{1}^{q}[index]$ is guaranteed to be optimal.
		
		\subsubsection{Accelerated lower-level problem solving}\label{Sec:Sovel_LL}
		
		The time-space network graph $\mathcal{G}$ is a complete graph with an exponential number of edges (Fig.~\ref{Fig:time_space}), which places a significant memory burden on the computer. However, most edges in the graph have a value of 0. For instance, a truck does not travel from sites in $\mathcal{P}$ to other sites in $\mathcal{P}$, from $\mathcal{S}$ to $\mathcal{S}$, or from $\mathcal{D}$ to $\mathcal{D}$; likewise, a production site $i \in \mathcal{S}$ only transports waste to a limited number of backfill sites $j \in \mathcal{D}$, or possibly just one. By removing these impossible edges (those with a flow rate of 0) to obtain support graph, the model size is greatly reduced. This reduction allows the commercial solver to provide an exact solution in a reasonable amount of time.
		
		It is important to note that for each $\mathbf{y}^k$, there may be multiple possible solutions for $\mathbf{x}^k$. However, in practice, the carrier is indifferent to pollution and tends to select an $\mathbf{x}^k$ randomly~\citep{liu2019expand}. This randomness may lead to sub-optimal values for $F_1(\mathbf{x}^k)$ and $F_2(\mathbf{x}^k, \mathbf{y}^k)$, which can influence the MOPSO algorithm when updating the particle positions~\citep{soares2020designing}. Therefore, incorporating a random selection strategy can enhance the robustness of the model. Section~\ref{Sec:Cp} discusses the accuracy of the proposed method.
		
		\subsection{Optimality gap estimation and algorithmic evaluation}\label{Sec:GAP}
		
		The upper and lower bounds of the model cannot be directly determined by the proposed hybrid approach and require further discussion. This section provides a detailed analysis of the upper and lower bounds of pollution and defines the utilization rate of government subsidies.
		
		\subsubsection{Gap of pollution index}
		
		We define $\mathbf{x}^*_{M1}$ as the optimal solution of model \textbf{[M1]}, which focuses on the carrier's profit while ignoring environmental pollution. The following proposition holds:
		
		\begin{proposition} $F_1(\mathbf{x}_{M1}^*)$ is the upper bound of the pollution index $F_1$ in model \textbf{[M3]}. \end{proposition}
		
		\begin{proof} Since $\mathbf{x}^*_{M1}$ is the optimal solution of model \textbf{[M1]}, it is also a feasible solution for model \textbf{[M3]}. Therefore, $F_1(\mathbf{x}_{M1}^*)$ serves as an upper bound for the pollution index in model \textbf{[M3]}. \end{proof} 
		
		$F_1(\mathbf{x}_{M1}^*)$ can be interpreted as the pollution generated by the carrier while transporting CW without any subsidy measures. Therefore, the rate of pollution reduction (RoPR) can be expressed as the gap between $F_1(\mathbf{x}_{M3}^*)$ and $F_1(\mathbf{x}_{M1}^*)$:
		
		\begin{equation}
			RoPR=\frac{F_1(\mathbf{x}_{M1}^*)-F_1(\mathbf{x}^*_{M3})}{F_1(\mathbf{x}_{M1}^*)}\times100\%,
		\end{equation}
		where $\mathbf{x}_{M3}^*$ represents the solution obtained by solving model \textbf{[M3]} using the proposed hybrid approach. If we require $\mathbf{x}$ to be feasible rather than optimal, we can relax the bi-level model \textbf{[M2]} into a single-level model, known as the \textit{high point problem}~\citep{moore1990mixed}:
		\begin{align}
			&\text{\bf{[M2-HPR]}}\notag\\ &\min
			\begin{aligned}[t]
				F_{1}(\mathbf{x})
			\end{aligned}\label{HPR:obj} \\
			&\text{subject to:} \notag \\
			&\qquad \text{Constraints}~\eqref{M1:start}-\eqref{M1:zero3},~\eqref{M2:PL},~\eqref{M2:PU}. \notag
		\end{align}
		
		We define $\mathbf{x}_{HPR}^*$ represents the optimal solution of model \textbf{[M2-HPR]}. The following proposition holds:
		
		\begin{proposition}
			$F_1(\mathbf{x}_{HPR}^*)$ is the lower bound of $F_1$ in model \textbf{[M3]}.
		\end{proposition}
		
		\begin{proof}
			$F_1(\mathbf{x}_{HPR}^*)$ is the lower bound that can be interpreted from two perspectives: 1) model \textbf{[M2-HPR]} is a relaxation of model \textbf{[M2]}. 2) The objective function~\eqref{HPR:obj} and constraints~\eqref{M1:start}-\eqref{M1:zero3} do not contain the decision variable $\mathbf{y}$, so constraints~$\eqref{M2:PL}$ and~\eqref{M2:PU} can be removed from the model \textbf{[M2-HPR]}. Observing models \textbf{[M2-HPR]} and \textbf{[M1]}, we can see that they represent the truck schedule developed by the carrier and the government under the same operating conditions, respectively.
		\end{proof}
		
		Therefore, the GAP of the objective function $F_1$ in model \textbf{[M3]} is calculated as follows:
		\begin{equation}
			GAP\_F_1=\frac{F_1(\mathbf{x}_{M3}^*)-F_1(\mathbf{x}_{HPR}^*)}{F_1(\mathbf{x}_{HPR}^*)}\times100\%,
		\end{equation}
		
		\subsubsection{Effective subsidy rate of government}
		
		To reduce pollution, the carrier needs to change trucks’ schedules. This will result in additional transportation costs, which should be compensated by the government. Although higher subsidies initially reduce pollution, their effect diminishes with further increases, eventually only raising the carrier’s profit. The following formula measures the relationship between the carrier’s profit and the government’s revenue with subsidizing:
		\begin{equation}
			|F_2(\mathbf{x}_{M3}^*)-F_2(\mathbf{x}_{M1}^*)|=|f^{\prime}(\mathbf{x}_{M3}^*,\mathbf{y}_{M3}^*)-f(\mathbf{x}_{M1}^*)|+ES,
		\end{equation}
		where the terms $-f(\mathbf{x}_{M1}^*)$ and $-F_2(\mathbf{x}_{M1}^*)$ represent the carrier's profit and the government's revenue, respectively, without any subsidy. These values are derived from solving model \textbf{[M1]}. When high-quality subsidies are applied (corresponding to model \textbf{[M3]}), the carrier's profit and the government's revenue are $-f^{\prime}(\mathbf{x}_{M3}^*,\mathbf{y}_{M3}^*)$ and $-F_2(\mathbf{x}_{M3}^*)$, respectively. Therefore, $|F_2(\mathbf{x}_{M3}^*)-F_2(\mathbf{x}_{M1}^*)|$, $|f^{\prime}(\mathbf{x}_{M3}^*,\mathbf{y}_{M3}^*)-f(\mathbf{x}_{M1}^*)|$, $ES$ represent the government subsidy, ineffective subsidy (i.e., increase in the carrier's profit), and effective subsidy, respectively.
		
		To avoid ineffective subsidies, we define the Effective Subsidy Rate (ESR) as an index to evaluate the performance of model \textbf{[M3]} in achieving the objective function $F_2$. The ESR ranges from $(0, 1]$, with values closer to 1 indicating a more effective subsidy scheme. The ESR is defined as follows:
		
		\begin{equation}\begin{aligned}
				ESR& =\frac{ES}{|F_2(\mathbf{x}_{M3}^*)-F_2(\mathbf{x}_{M1}^*)|}\times100\%  \\
				&=\frac{|F_{2}(\mathbf{x}_{M3}^{*})-F_{2}(\mathbf{x}_{M1}^{*})|-|f^{\prime}(\mathbf{x}_{M3}^{*},\mathbf{y}_{M3}^{*})-f(\mathbf{x}_{M1}^{*})|}{|F_{2}(\mathbf{x}_{M3}^{*})-F_{2}(\mathbf{x}_{M1}^{*})|}\times100\%.
		\end{aligned}\end{equation}
		
		\section{Case study in Chengdu}\label{Sec:Cs}
		
		\subsection{Experiment settings}\label{Sec:parameter_set}
		
		We conduct numerical experiments using real data from Chengdu. Chengdu is one of China's new first-tier cities with a resident population of over 20 million. Currently, there are more than 15,000 trucks in Chengdu city. Fig.~\ref{Fig:trucks} shows the frequency distribution of the number of trucks owned by enterprises, with most enterprises owning fewer than 100 trucks. According to study~\cite{HAN2023}, we consider a large case of 240 trucks that serves 30 sites in Longquanyi District, Chengdu (see Fig.~\ref{Fig:opreation_area}), including 17 production sites, 10 backfill sites, and 3 processing facilities. We review government statistical reports and interview the carrier to collect the parameters for the experiment. There are mainly 4 sets of parameters in use, i.e., travel time-related, cost-related parameters, pollution-related, and hybrid approach-related. The values of these parameters are depicted in~\ref{Appendix:parameters} and are explained below:
		
		\begin{figure}[htbp]
			\centering
			\includegraphics[width=0.5\textwidth]{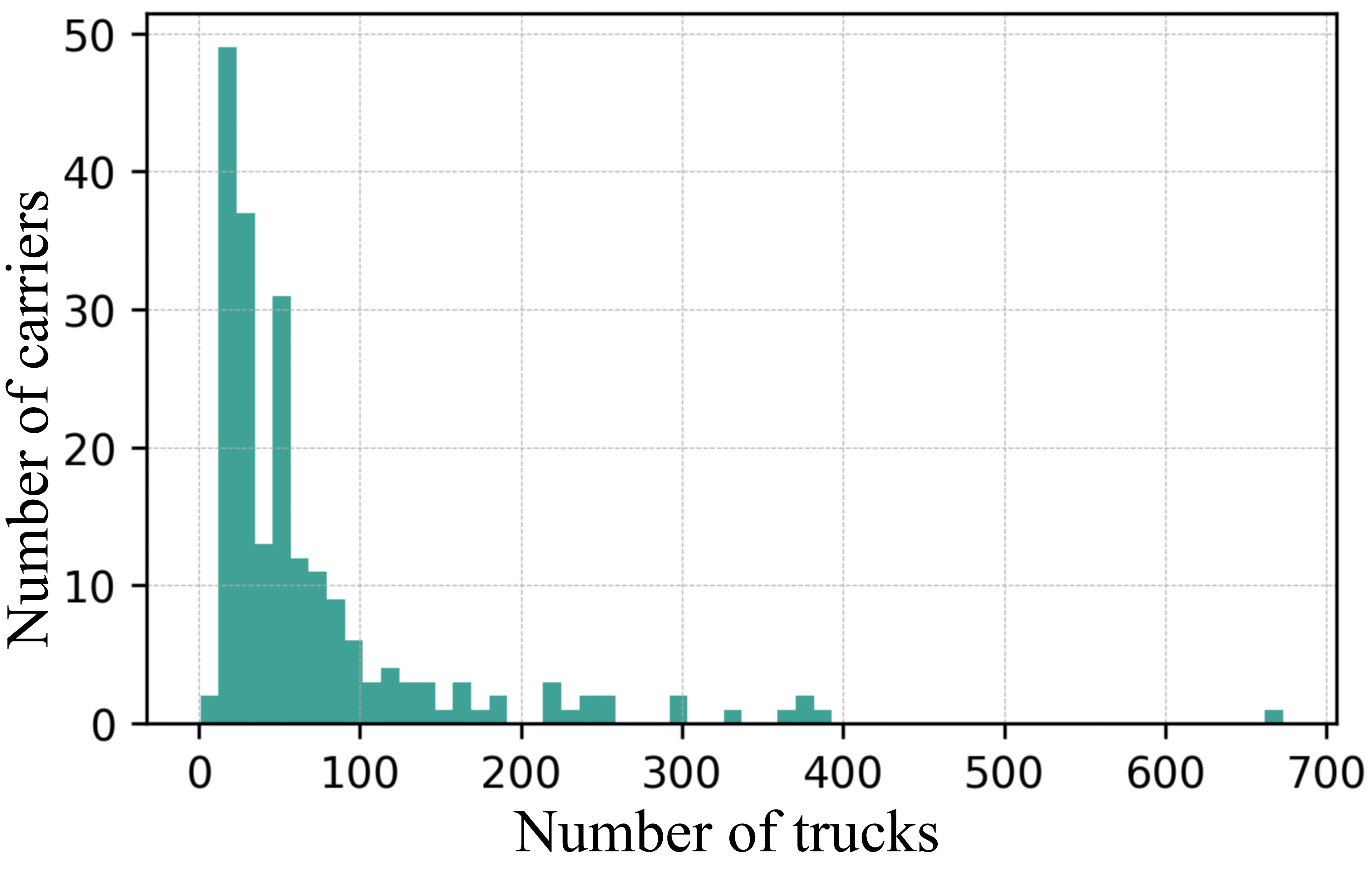}
			\caption{\label{Fig:trucks}Frequency distribution of trucks held by carriers}
		\end{figure}
		
		The distance $d$ between each site is obtained via the~\cite{amapAPI}. We take a planning period of 10 hours with a time interval $\Delta t=10$ minutes, resulting in $T=60$ planning periods. It is difficult to accurately measure the traveling time between sites due to the uncertainty of traffic conditions on the road, so it is usually approximated by the average state in practice~\citep{chu2012optimization,yazdani2021improving}. The average speed of a truck traveling in Chengdu is set as $u=30$ kilometer per hour. Therefore, the travel time between sites is calculated as $\overline{r}_{i,j}=\frac{d_{i,j}}{u}$ minutes, $i,j\in\{0\}\cup\mathcal{P}\cup\mathcal{S}\cup\mathcal{D}$. To ensure sufficient time redundancy, we set the traveling time as an integer multiple of the time interval $\Delta t$ and round upwards, i.e. $r_{i,j} = \left\lceil \frac{\overline{r}_{i,j}}{\Delta t} \right\rceil$, so the number of planning periods is now approximately $\overline{T}=\max\{r_{i,j}\}+1$, $i,j\in\{0\}\cup\mathcal{P}\cup\mathcal{S}\cup\mathcal{D}$. The time to load/unload a truck load of CW is usually 2-5 minutes, depending on the uncertain operating conditions. By \citeauthor{chu2012optimization}, we set the maximum number of load/unload operations per time interval $\Delta t$ at the production site, backfill site, and processing facility to 2, 3, and 3, respectively.
		
		We consider both diesel and electrical truck fleets. \cite{HAN2023} explores the cost components of trucks in detail. For a diesel truck with an age of 5 years, fixed cost = purchase cost + maintenance cost - residual value = 987,000 + 149,000 - 185,000 = 951,000 CNY, then the average daily fixed cost = 951,000/365/5 = 521 CNY. Adding the driver's salary, we set the fixed cost of the daily operation of diesel trucks as $C_{0,v} = 750~\text{CNY},v\in\mathcal{F}_d$. For an electrical truck, the purchase cost, maintenance cost, and residual value during the operation period of 5 years are 521,000, 196,000, and 100,000 CNY, respectively. The driver's salary is the same as that of the diesel truck, so the daily fixed cost of the electrical truck is calculated as $C_{0,v} = 550~\text{CNY},v\in\mathcal{F}_e$. The diesel price in Chengdu city is 7.8 CNY per liter, and the truck consumes 50 liters of fuel for 100 kilometers, so the cost of traveling for a time interval $\Delta t$ is calculated as $C_{1,v}=19.5~\text{CNY},v\in\mathcal{F}_d$. The night-time price of electricity is 0.4 CNY per kilowatt hour (kWh), while the daytime price is 1.5 CNY per kWh. We therefore take the average price of electricity as 0.95 CNY per kWh. The electrical truck consumes 200 kWh for 100 kilometers, so the cost of traveling for a time interval $\Delta t$ is calculated as $C_{1,v}=9.5~\text{CNY},v\in\mathcal{F}_e$. The transport price and treatment fee of CW are dynamically adjusted in practice, and there are no robust pricing standards in Chengdu. Through consultation with practitioners in the construction industry, it is known that the average transport price of CW is 20-30 CNY per tonne, and the treatment fee is 3-7 CNY per tonne. Therefore, we set the transport price $C_2=25$ CNY per tonne, and the market price for treatment fee $y^{\prime}=5$ CNY per tonne. The volume of CW produced is difficult to obtain, so we make reasonable assumptions based on publicly available data. According to the government statistics report~\citep{chengdu_waste_plan_2024}, the Longquanyi District in Chengdu City produces approximately 180,000 tonnes of CW per day, with each production site generating an average of 900 tonnes of CW. Therefore, we set the daily production and backfill from construction sites as a normal distribution with a mean of 900 and a standard deviation of 100.
		
		For further details on the parameters of electric and diesel trucks, refer to ~\cite{NCWHT} and \cite{DCWHT}, respectively. Emission model parameters are from \cite{barth2005development} and presented in Table~\ref{Tab:parameter_value}. We assume the acceleration $\tau$ and road angle $\delta$ are both 0 based on regional traffic conditions\citep{barth2005development, DEMIR2011347}. Through consultation with practitioners in the construction industry, we set up three processing facilities to treat one tonne of CW resulting in pollution indices of $h_1=0.2,h_2=0.4,h_3=0.6$ respectively.
		
		After a large number of repeated experiments, we found that the following values of the parameters will lead to quick convergence. $\omega_0=0.8,\omega_1=0.1,\omega_2=0.1,M=200,m=10,p_m=0.2,G^{\prime}=3,\sigma=0.2,\epsilon=0.001$. The number of particles $K=40$. The upper-level objective function will not further improve usually after 20 iterations.		
		
		\begin{figure}[htbp]
			\centering
			\includegraphics[width=0.45\textwidth]{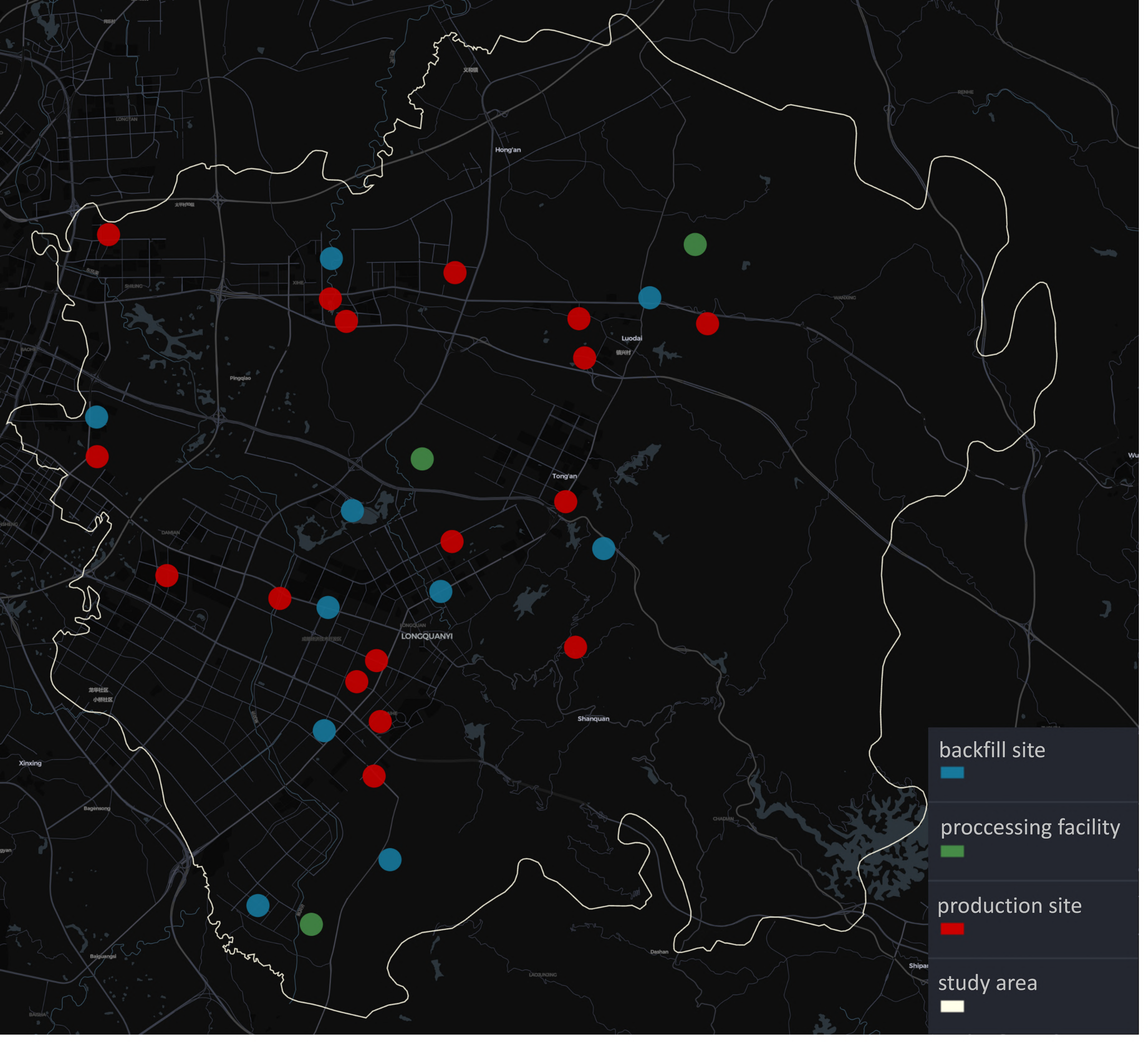}
			\caption{\label{Fig:opreation_area}The location of sites.}
		\end{figure}
		
		\subsection{Computational performance}\label{Sec:Cp}
		
		According to the parameter settings in Section~\ref{Sec:parameter_set}, the lower level model $\bf{[M1]}$ contains 77,556 integer variables and 5,548 constraints. We implement the proposed hybrid approach on a laptop equipped with Intel(R) Core(TM) i9-10850K CPU at 3.6 GHz and 128 GB of RAM on a Windows 11 64-bit OS. We then discuss the model's performance for government and carrier separately.
		
		\subsubsection{Optimal strategy of the carrier}
		
		We solve model \textbf{[M1]} using the Gurobi solver~\citep{gurobi}, setting a time limit of 600 seconds. To assess the model's performance, we run 10 randomized experiments, with the results illustrating the variation of the model gap over time, as shown in Fig.~\ref{Fig:RunTime}. We observe that a high-quality solution (with a gap of less than 1\%) can be obtained quickly, typically within 15 seconds. This is because the optimal solution of the original problem is close to the optimal solution of its linear relaxation. However, achieving a solution with a gap of less than 0.1\% proves to be more time-consuming.
		
		\begin{figure}[!t]
			\centering
			\includegraphics[width=0.5\textwidth]{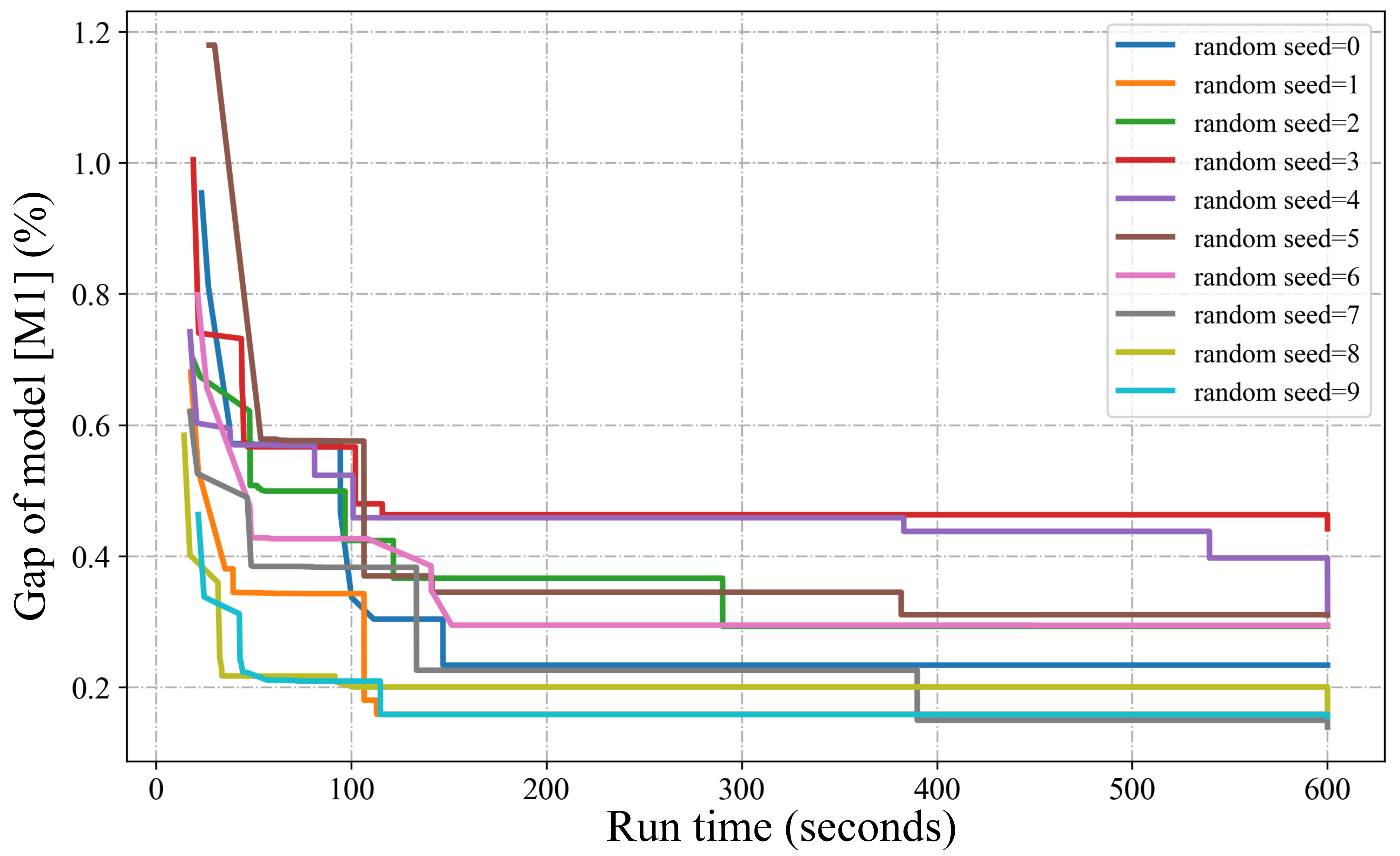}
			\caption{\label{Fig:RunTime}The variations of model gap with run time. Ten repetitions of the experiment for model \textbf{[M1]}.}
		\end{figure}
		
		In the experiment, we set the random seed to 42 and aim for a gap of less than 0.5\%. The results are presented in Table~\ref{Tab:M1_results}. In this scenario, the total amount of CW at the production sites exceeds the demand at the backfill sites. As a result, trucks are tasked with transporting waste from the production sites to the processing facilities and backfill sites, but not from the processing facilities to the backfill sites. A total of 105 trucks are in operation, consisting of 1 electric truck and 104 diesel trucks. As mentioned in Section~\ref{Sec:Problem_f}, the carrier prefers using diesel trucks to minimize costs. The electric truck is only deployed when the diesel trucks reach full capacity.
		
		\begin{table}[h!]
			\centering
			\caption{Results of model \textbf{[M1]}}\label{Tab:M1_results}
			\begin{tabular}{l r}
				\toprule
				Parameters & Values \\
				\midrule
				Objective value: negative value of the carrier's profit (CNY) & -230,242.50 \\
				Gap & 0.37\% \\
				Runtime (second) & 45 \\
				Revenue of government (CNY) & 35,375.00 \\
				Pollution index & 11,723.23 \\
				Number of electrical trucks & 1 \\
				Number of diesel trucks & 104 \\
				Transport volume of electrical trucks (tonne) & 70 \\
				Transport volume of diesel trucks (tonne) & 15,570 \\
				Transport volume from $\mathcal{S}$ to $\mathcal{P}$ (tonne) & 7,075 \\
				Transport volume from $\mathcal{S}$ to $\mathcal{D}$ (tonne) & 8,565 \\
				Transport volume from $\mathcal{P}$ to $\mathcal{D}$ (tonne) & 0 \\
				Total volume of CW (tonne) & 15,640 \\
				\bottomrule
			\end{tabular}
		\end{table}
		
		Although the ideal method for evaluating the model's performance would involve comparing its results to current practices, complete real-world data is not available. Based on the analysis in Section~\ref{Sec:parameter_set}, we estimate that the carrier requires 240 trucks to complete the transportation task. In contrast, the proposed approach only requires 105 trucks. This suggests that our method enhances transportation efficiency by 2 to 3 times.
		
		\subsubsection{Optimal strategy of the government}
		
		\begin{figure}[!t]
			\centering
			\includegraphics[width=0.5\textwidth]{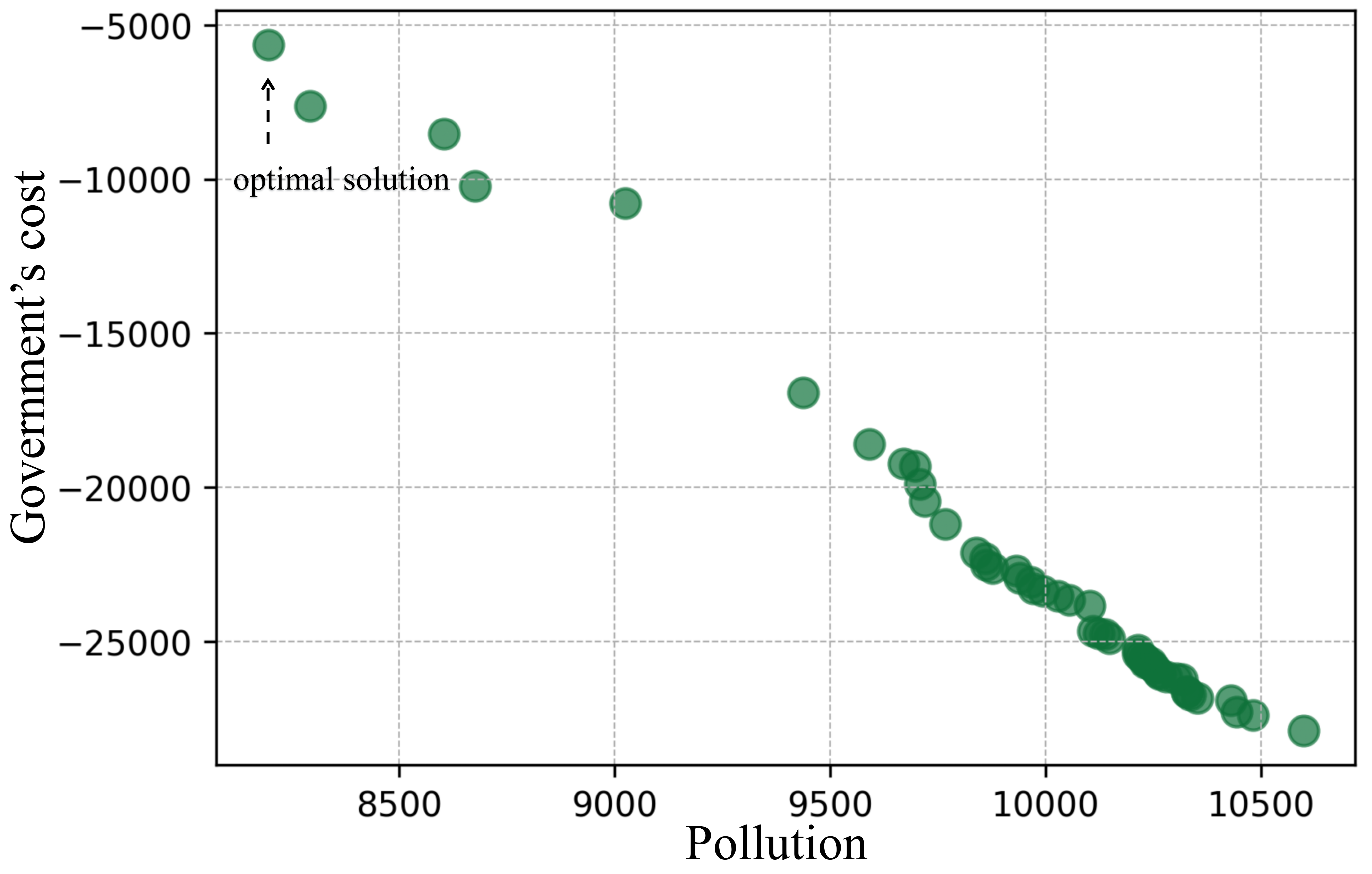}
			\caption{\label{Fig:M3_iteration}The Pareto optimal solution set of \textbf{[M3]}.}
		\end{figure}
		
		\begin{table}[h!]
			\centering
			\caption{Results of model \textbf{[M2-HPR]}}
			\label{Tab:M2HPR_results}
			\begin{tabular}{l r}
				\toprule
				Parameters & Values \\
				\midrule
				Objective value: pollution index & 8,142.39 \\
				Gap & 0.43\% \\
				Runtime (second) & 58 \\
				Profit of carrier (CNY) & 160,994 \\
				Revenue of government (CNY) & 35,175 \\
				Number of electrical trucks & 30 \\
				Number of diesel trucks & 210 \\
				Transport volume of electrical trucks (tonne) & 2,790 \\
				Transport volume of diesel trucks (tonne) & 12,825 \\
				Transport volume from $\mathcal{S}$ to $\mathcal{P}$ & 7,035 \\
				Transport volume from $\mathcal{S}$ to $\mathcal{D}$ & 8,580 \\
				Transport volume from $\mathcal{P}$ to $\mathcal{D}$ & 0 \\
				Total volume of CW (tonne) & 15,615 \\
				\bottomrule
			\end{tabular}
		\end{table}
		
		To assess the solution quality for model \textbf{[M3]}, we begin by solving model \textbf{[M2-HPR]} using the Gurobi solver. The solution of the high-point model represents the system's performance if all truck schedules are coordinated by the government. Under identical transportation tasks, the results in Table~\ref{Tab:M2HPR_results} show a 30.54\% reduction in the pollution index ($\frac{11,723.23-8,142.39}{11,723.23}\times 100\%=30.54\%$) in the HPR compared to model \textbf{[M1]}. However, this improvement in pollution control comes at a significant cost: the carrier's profit decreases substantially, from 230,242.50 CNY to 160,994.00 CNY.
		
		We then apply the proposed method to solve model \textbf{[M3]}. In each run, we solve 800 integer programming subproblems (40 particles $\times$ 20 iterations). We set a GAP of 1\% for each subproblem, and the total running time for the proposed method is 13,536 seconds (approximately 3.76 hours). Given that each planning period spans 10 hours, this solution time is acceptable. Fig.~\ref{Fig:M3_iteration} illustrates the Pareto solution set of model \textbf{[M3]}. An interesting ``jump'' phenomenon is observed, where few feasible and Pareto optimal solutions for the pollution index lie in the interval $[8500,9500]$. This occurs when the subsidy amount is insufficient to cover the additional cost of using electric trucks, leading the carrier to favor diesel trucks. However, once the subsidy surpasses this threshold, electric trucks become more attractive to the carrier than diesel trucks.
		
		The high-quality solution of model \textbf{[M3]} is the Pareto optimal solution with the minimum pollution index in Fig.~\ref{Fig:M3_iteration}. The results are shown in Table~\ref{Tab:M3}. The pollution index (primary objective function value) is 8,265.56, so $GAP\_F_1 = 1.51\%$ ($\frac{8,265.56-8,142.39}{8,142.39}\times 100\%=1.51\%$). It shows that the proposed method finds a satisfactory solution for the pollution index. The cost of the government (secondary objective function value) is 6,655.17, thus the effective subsidy rate $ESR = 95.74\%$ ($\frac{|6,655.17-(-35,375.00)|-|(-232,032.87)-(-230,242.50)|}{|6,655.17-(-35,375.00)|}\times 100\%=95.74\%$), i.e., 95.74\% of the government subsidy amount is used to reduce pollution. It shows that the proposed method finds a satisfactory solution for government subsidies. Therefore, we believe that the quality of the solution found by the proposed method is high and close to the global optimal solution.
		
		\begin{table}[h!]
			\centering
			\caption{Results of model \textbf{[M3]}}\label{Tab:M3}
			\begin{tabular}{l r}
				\toprule
				Parameters & Values \\
				\midrule
				Primary objective $F_1$: pollution index & 8,265.56 \\
				Secondary objective $F_2$: cost of government (CNY) & 6655.17 \\
				Runtime (second) & 13,536 (3.76 hours) \\
				Profit of carrier (CNY) & 232,032.87 \\
				Number of electrical trucks & 30 \\
				Number of diesel trucks & 73 \\
				Transport volume of electrical trucks (tonne) & 2,760 \\
				Transport volume of diesel trucks (tonne) & 12,820 \\
				Transport volume from $\mathcal{S}$ to $\mathcal{P}$ (tonne) & 7,025 \\
				Transport volume from $\mathcal{S}$ to $\mathcal{D}$ (tonne) & 8,555 \\
				Transport volume from $\mathcal{P}$ to $\mathcal{D}$ (tonne) & 0 \\
				Total volume of CW (tonne) & 15,580 \\
				The gap of objective function $F_1$, $GAP\_{F_1}$ & 1.51\% \\
				Effective subsidy rate, $ESR$ & 95.74\% \\
				\bottomrule
			\end{tabular}
		\end{table}
		
		According to Table~\ref{Tab:M3}, the carrier dispatches all 30 electrical trucks to transport 2,760 tonnes of waste and an additional 73 diesel trucks for 12,820 tonnes. The government subsidizes 42,000 CNY on the waste transportation process, which will reduce pollution by 29.49\% ($\frac{11,723.23-8265.56}{11,723.23}\times 100\%=29.49\%$). To further explore the operation regulations of the trucks, we examine Fig.~\ref{Fig:M3_movement}. We observe some interesting phenomena: (1) Most production sites prioritize transporting waste to the nearest backfill sites. (2) When the closest site is occupied, the production site will opt for the second nearest backfill site. For instance, production site $S_1$ initially chooses backfill site $D_1$, but since $D_1$ is already assigned to handle waste from $S_2$, the waste from $S_1$ is instead shipped to $D_2$ to reduce waiting time. This scheduling approach helps reduce congestion compared to manual methods. (3) The subsidy can incentivize the carrier to use electric trucks for long-distance waste transportation. For example, transportation to $P_1$ remain unchanged with or without the subsidy, but transportation to $P_2$ and $P_3$ shift from diesel trucks to electric trucks. This change occurs because $P_2$ and $P_3$ are further from surrounding production sites, making electric trucks a more economical option. Further analysis of important parameters is presented in Section~\ref{Sec:SA}.
		
		\begin{figure}[!t]
			\centering
			\includegraphics[width=0.8\textwidth]{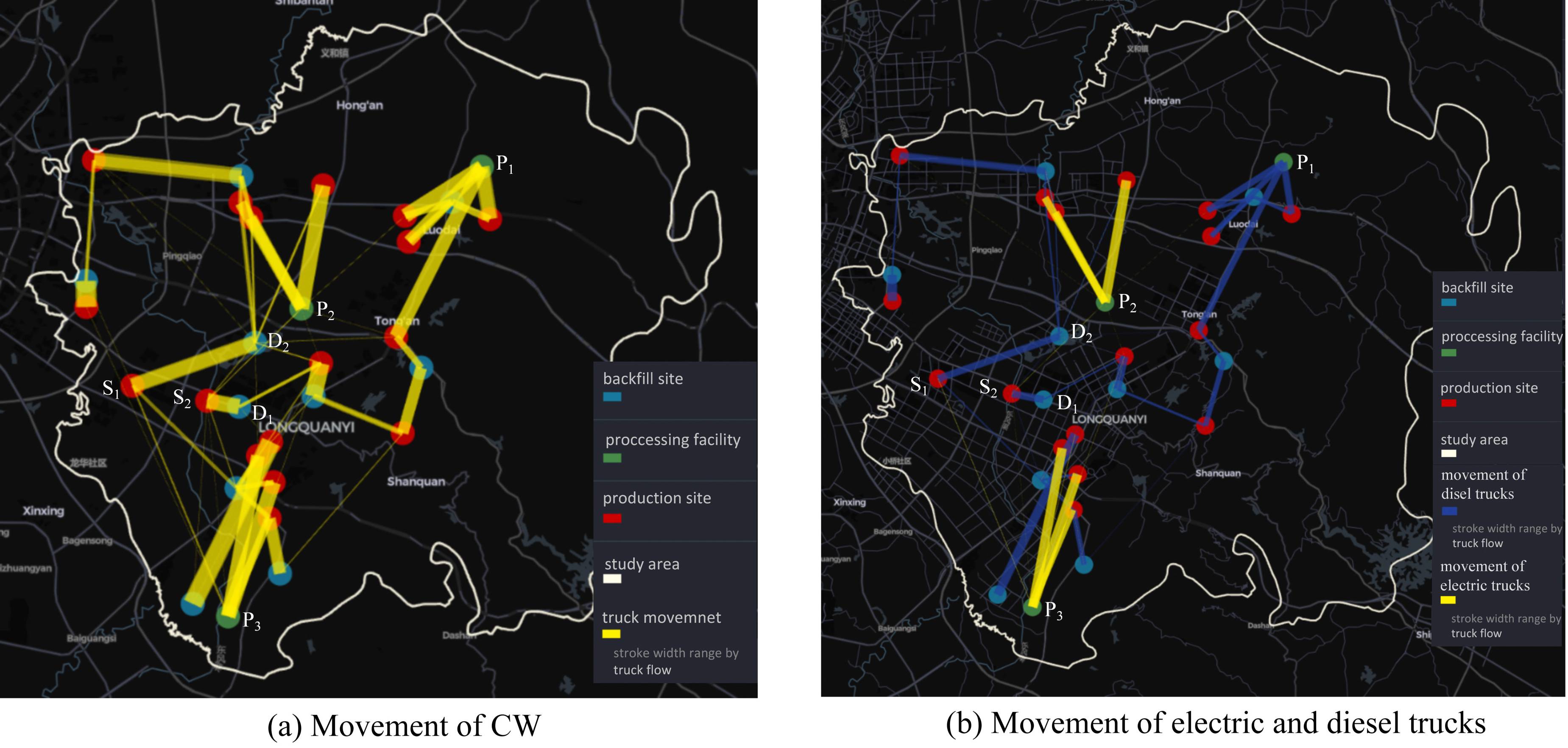}
			\caption{\label{Fig:M3_movement}Movement of CW and trucks between sites. (a) illustrates the flow of CW between each site in a planning period, where the stroke width of the arcs represents the amount of CW. (b) shows the movement of electrical and diesel trucks between each site in a planning period, where the stroke width of the arcs corresponds to the total number of trips.}
		\end{figure}
		
		\subsection{Sensitivity analysis}\label{Sec:SA}
		
		To understand the influence of the model parameters on the solution, a sensitivity analysis is conducted on the key parameters of the model, including the bounds of the treatment fee and the number of electrical trucks. 
		
		\subsubsection{Impact of the treatment fee}
		While the other parameters are fixed, we first examine the impact of the bounds on the treatment fee on the system performance. Due to the fluctuation of the treatment fee with the market and policy, its impact on environmental pollution, government subsidies, and carrier profit may change accordingly~\citep{huang2018construction}. Thus, we test the performance of the proposed model for a series of  $PU$s and $PL$s. The values of $PU$s are drawn from the range of [3,7] with a step of 0.5 and $PL=PU-10$ for each lower bound. The experimental results are shown in Fig.~\ref{Fig:treatment_fee}. The fluctuation of the pollution index is very small in the 9 experiments, which indicates that the proposed method can produce a robust solution for all $[PL,PU]$ (see Fig.~\ref{Fig:treatment_fee} (a)). Interestingly, the pollution index decreases slowly with the increase of $PU$. That’s because the carrier tends to strategically decline long distance transport tasks when the treatment fee is too high. Fig.~\ref{Fig:treatment_fee} (b) shows that the government's revenue and carrier's profit are negatively correlated. In the real world, the government can refer to Fig.~\ref{Fig:treatment_fee} (b) and flexibly set the interval $[PL, PU]$ of treatment fee according to its financial and the management situation of carriers.
		
		\begin{figure}[!t]
			\centering
			\includegraphics[width=0.8\textwidth]{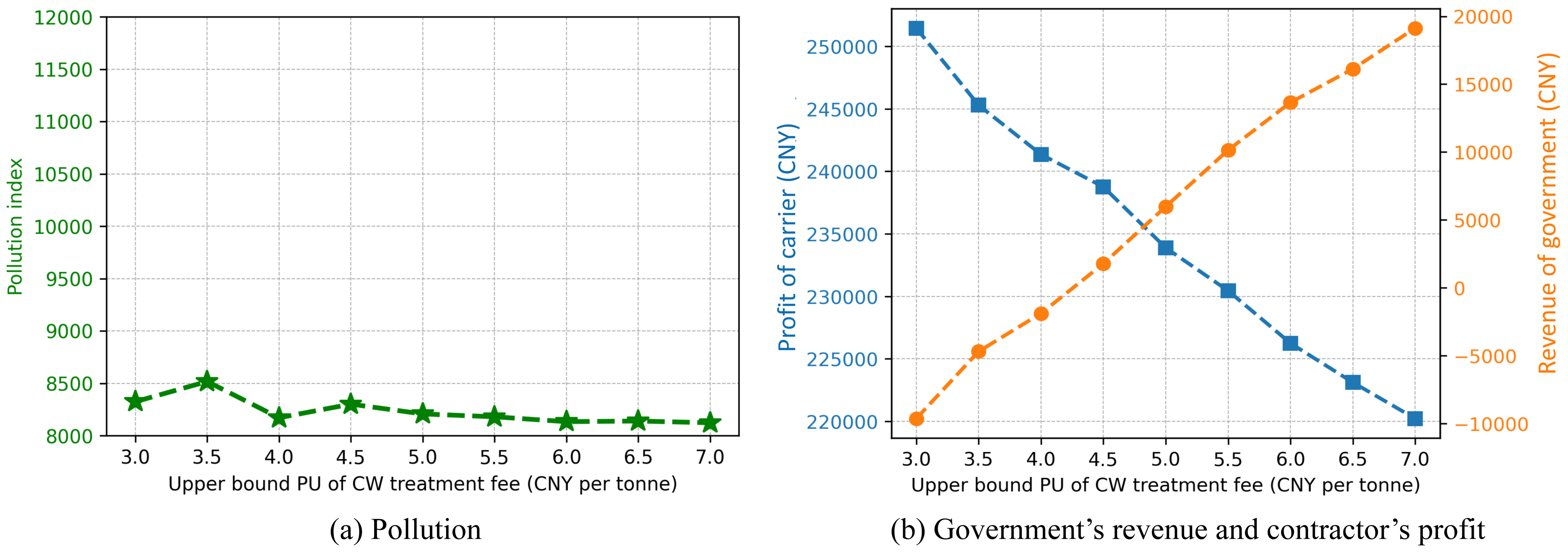}
			\caption{\label{Fig:treatment_fee}Sensitivity of the treatment fee.}
		\end{figure}
		
		\subsubsection{Impact of the number of electrical trucks}
		
		Increasing the ratio of electrical trucks to the CW transportation system could potentially reduce emissions though at the price of higher cost. Based on the parameters in Section~\ref{Sec:parameter_set}, we set the number of electrical trucks  $N_1$ to [10, 50] with a step size of 10 and evaluate its effects. 
		
		Fig.~\ref{Fig:number_trucks} demonstrates the changes in the pollution index and government revenue with different ratios of electrical trucks. Both the pollution index and government revenue show an approximately linear relation with the number of trucks. The pollution index is reduced by 1,059.73 or  9.04\% ($\frac{1059.73}{11723.23}\times100\%=9.04\%$) with a 10\% increase of electrical trucks. The government subsidy, on the other hand, increases by an average of 10,942.50 (CNY). Therefore, increasing the number of electrical trucks can significantly reduce pollution though may incur higher costs.
		
		\begin{figure}[htbp]
			\centering
			\includegraphics[width=0.5\textwidth]{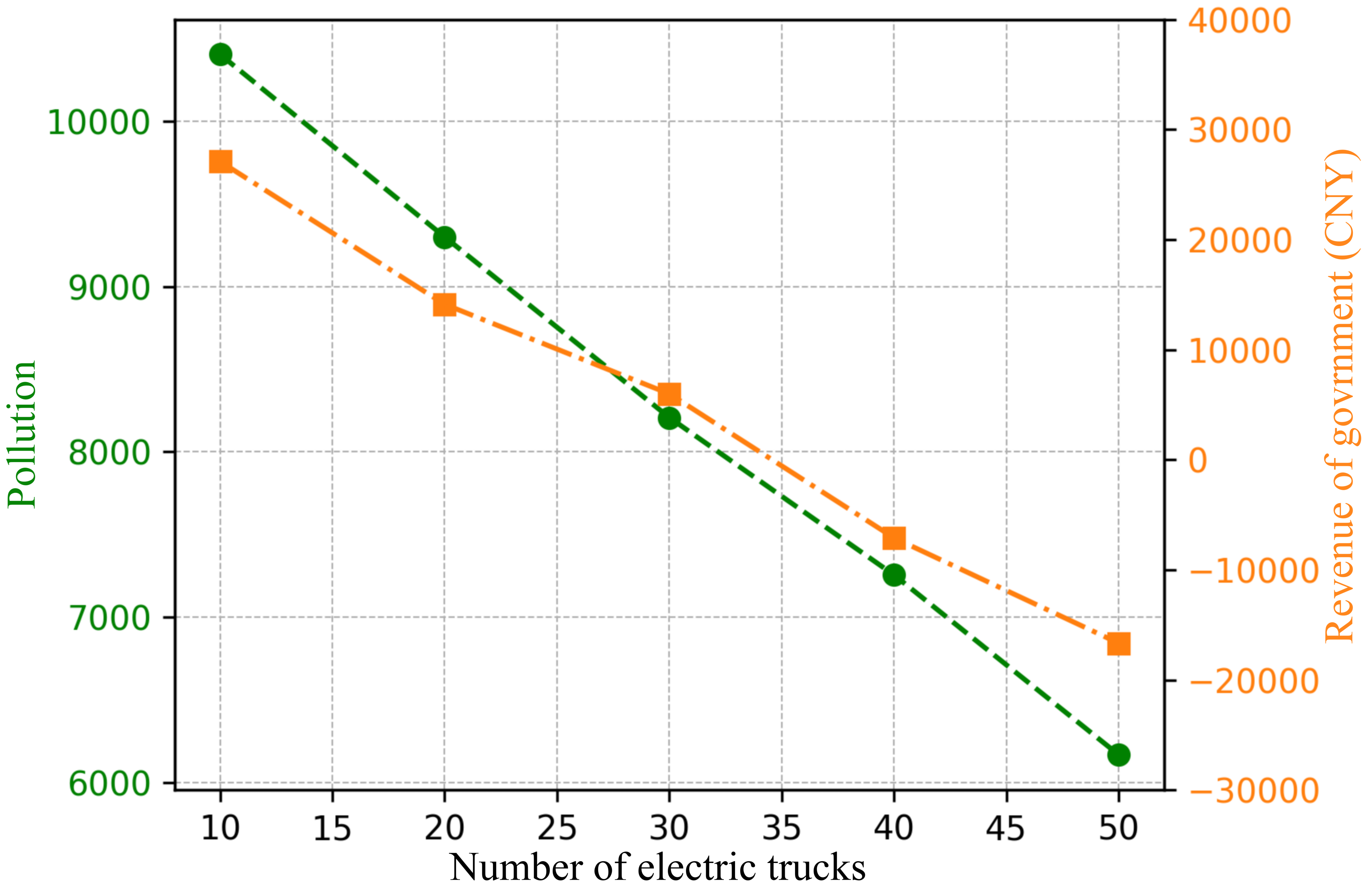}
			\caption{\label{Fig:number_trucks}Sensitivity of the number of electrical trucks.}
		\end{figure}
		
		\section{Conclusions}\label{Sec:Conclusions}

		This paper examines the operation of construction waste hauling trucks, focusing on the interaction between the government and the carrier. We first formulate a tailored multi-vehicle minimum cost flow model \textbf{[M1]} to maximize carrier profit without subsidies. Then, we explore the Stackelberg game between the government and the carrier, where the government influences the carrier’s schedule through treatment fees. A bi-level optimization model \textbf{[M3]} is developed to optimize the government’s dual objectives: minimizing pollution and subsidy expenditures,where model \textbf{[M1]} serves as the lower-level problem. To solve the challenging bi-level model, we propose a hybrid MOPSO-based approach, solving \textbf{[M1]} for each particle with a commercial solver. A large-scale case study in Chengdu shows that the hybrid method achieves a high-quality solution ($GAP\_F\_1 = 1.51\%, ESR = 95.74\%$) in a reasonable time (3.76 hours). Results indicate that appropriate subsidies can reduce pollution by 29.49\%. Insights into the truck scheduling problem are also provided.
		
		1) The proposed approach can improve truck efficiency by 2-3 times. Without subsidies, carriers prefer diesel trucks, but with subsidies, electric trucks are favored, especially for longer trips.
		
		2) The carrier's profit is positively correlated with the government subsidy. By adjusting the treatment fees strategically, the government can intervene in the carrier's transportation schedule, as shown in Fig.~\ref{Fig:treatment_fee}.
		
		3) Increasing the number of electrical trucks can significantly reduce environmental pollution, but their high cost requires government subsidies to incentivize the shift from diesel to electric trucks.
		
		The model is flexible and can be extended to accommodate more complex scenarios. For example,~\ref{Appendix:models} extends the model to include multiple types of CW. While the models show potential for emission reduction, they are limited by fixed travel speeds, which may not reflect real-world conditions under uncertainty. A potential improvement is to develop a robust optimization model that accounts for input parameter stochasticity. However, data collection may be challenging due to carriers' reluctance to share operational details, in which case a simplified model could be developed.

		\bibliographystyle{unsrtnat}
		\bibliography{ref.bib}
		\clearpage
		\appendix
			\section{Table of notations}\label{Appendix:Notation}
			
			\setcounter{table}{0}   
			\setcounter{figure}{0}
			\setcounter{equation}{0}
			\renewcommand{\thetable}{A\arabic{table}}
			\renewcommand{\thefigure}{A\arabic{figure}}
			\renewcommand{\theequation}{A\arabic{equation}}
			
			\begin{longtable}{p{0.08\textwidth} p{0.88\textwidth}}
				\caption{Summary of notation.}\\
				\toprule
				\underline{\textbf{\emph{Sets}}} & \\
				$\{0\}$ & Depot\\
				$\mathcal{P}$ & Processing facilities $\mathcal{P}=\{1,...,P\}$\\
				$\mathcal{S}$ & Production sites $\mathcal{S}=\{P+1,...,P+S\}$\\
				$\mathcal{D}$ & Backfill sites $\mathcal{D}=\{P+S+1,...,P+S+D\}$\\
				$\mathcal{T}$ & Planning period $\mathcal{T}=\{0,1,...,T\}$\\
				$\overline{\mathcal{T}}$ & Virtual planning period $\overline{\mathcal{T}}=\{-1,...,-\overline{T}\}$, $\overline{T}=\max\{r_{i,j}\}+1$\\
				$\mathcal{G}$ & Time-space network $\mathcal{G}=\{\mathcal{N},\mathcal{A}\}$, where $\mathcal{N}$ is the set of nodes and $\mathcal{A}$ is the set of arcs\\ 
				$\mathcal{A}^f$, $\mathcal{A}^d$ & Set of fully loaded arcs, set of deadheading arcs\\
				$\mathcal{A}^s$, $\mathcal{A}^0$ & Set of service arcs, set of zero arcs\\
				$\mathcal{F}_e$ & Electrical truck fleets $\mathcal{F}_e=\{\mathcal{V}_1,...,\mathcal{V}_E\}$\\
				$\mathcal{F}_d$ & Diesel truck fleets $\mathcal{F}_d=\{\mathcal{V}_{E+1},...,\mathcal{V}_{E+\overline{D}}\},E+\overline{D}=V$\\
				\underline{\textbf{\emph{Parameters}}} & \\
				$N_v$ & Total number of trucks in feet $v\in\mathcal{F}_e\cup\mathcal{F}_d$ (veh)\\
				$Q_v$ & Rated load weight per truck in fleet $v\in\mathcal{F}_e\cup\mathcal{F}_d$ (kilogram/veh)\\
				$\overline{Q}_v$ & Unloaded weight per truck in fleet $v\in\mathcal{F}_e\cup\mathcal{F}_d$ (kilogram/veh)\\
				$M_v$ & Total weight per truck in the feet $v\in\mathcal{F}_d$, $M_v=Q_v+\overline{Q}_v$ (kilogram/veh)\\
				$qs_i$ & Total weight of CW at production site $i\in\mathcal{S}$ during the planning period (tonne)\\
				$qd_j$ & Total weight of CW at backfill site $j\in\mathcal{D}$ during the planning period (tonne)\\
				$\Delta t$ & Time interval\\
				$r_{i,j}$ & Trucks take $r_{i,j}$ time intervals to travel from site $i$ to area $j$ for $i,j\in\{0\}\cup\mathcal{P}\cup\mathcal{S}\cup\mathcal{D}$\\
				$B_j$ & Maximum number of trucks allowed to be serviced during the time interval $\Delta t$ in site $j\in\mathcal{P}\cup\mathcal{S}\cup\mathcal{D}$\\
				$C_{0,v}$ & Fixed cost per truck in fleet $v\in\mathcal{F}_e\cup\mathcal{F}_d$ (CNY/veh)\\
				$C_{1,v}$ & Average cost of driving a time interval $\Delta t$ with unloaded and fully loaded trucks in fleet $v\in\mathcal{F}_e\cup\mathcal{F}_d$ (CNY)\\
				$C_2$ & Price of transporting a tonne CW (CNY/tonne)\\
				$y^{\prime}$ & Market guide price of treatment fee (CNY/tonne)\\
				$d_{i,j}$ & Distance between site $i\in\{0\}\cup\mathcal{P}\cup\mathcal{S}\cup\mathcal{D}$ and $j\in\{0\}\cup\mathcal{P}\cup\mathcal{S}\cup\mathcal{D}$ (meter)\\
				$u$ & Speed of trucks (meter/second)\\
				$PU$ & Upper bound of treatment fee(CNY/tonne) \\
				$PL$ & Lower bound of treatment fee (CNY/tonne) \\
				$\xi$ & Fuel-to-air mass ratio \\
				$k$ & Engine friction factor (kilojoule/revolution/liter)\\
				$N$ & Engine speed (revolution/second)\\
				$V$ & Engine displacement (liter)\\
				$\kappa$ & Heating value of a typical diesel fuel (kilojoule/gram)\\
				$\eta$ & Efficiency parameter for diesel engines\\
				$\eta_{t}$ & Drive train efficiency of trucks\\
				$\tau$ & Acceleration of trucks (meter/square second)\\
				$g$ & Gravitational acceleration (meter/square second)\\
				$\delta$ & Road angle (degree)\\
				$f_a$ & Coefficient of aerodynamic drag\\
				\multicolumn{2}{r}{\textit{(continued on next page)}} \\ 
				\bottomrule
			\end{longtable}
			
			\begin{longtable}{p{0.08\textwidth} p{0.88\textwidth}}
				\toprule
				\multicolumn{2}{l}{\textit{(continued)}} \\ 
				$f_r$ & Coefficient of rolling resistance\\
				$\varphi$ & Air density (kilogram/square meter)\\
				$A$ & Frontal surface area of diesel CHTH (square meter)\\
				$\theta$ & Conversion factor of fuel from gram/second to liter/second\\
				$h_j$ & Pollution factor, $j\in\mathcal{P}$\\
				$\lambda,\gamma,\alpha,\beta$ & $\lambda=\frac{\xi}{\kappa\theta}$, $\gamma=\frac{1}{1,000\eta_t\eta}$, $\alpha=\tau+g\sin\delta+gf_r\cos\delta$, $\beta=0.5f_{a}\varphi A$\\
				\underline{\textbf{\emph{Decision variables}}} & \\ 
				$x_{i,j,v,t}$ & Total flow of feet $v\in\mathcal{F}_e\cup\mathcal{F}_d$ from $i\in\{0\}\cup\mathcal{P}\cup\mathcal{S}\cup\mathcal{D}$ to $j\in\{0\}\cup\mathcal{P}\cup\mathcal{S}\cup\mathcal{D}$ at time $t\in\mathcal{T}\cup\overline{\mathcal{T}}$ (veh)\\ 
				$y_{j,v}$ & Treatment fee of transporting CW via fleet $v\in\mathcal{F}_e\cup\mathcal{F}_d$ to processing facility $j\in\mathcal{P}$ for harmless treatment (CNY/tonne) \\
				\bottomrule
			\end{longtable}
			
			\section{Table of partial parameter values}\label{Appendix:parameters}
			
			\setcounter{table}{0}   
			\setcounter{figure}{0}
			\setcounter{equation}{0}
			\renewcommand{\thetable}{B\arabic{table}}
			\renewcommand{\thefigure}{B\arabic{figure}}
			\renewcommand{\theequation}{B\arabic{equation}}
			\begin{table}[htb]
				\centering
				\caption{Partial parameter values.}
				\label{Tab:parameter_value}
				\begin{tabular}{l l | l l}
					\hline
					Parameter & Value & Parameter & Value\\
					\hline
					$T$ & 60 & $\Delta t$ & 10 minutes \\
					$u$ & 30 kilometers/hours & $C_{0,v}$ & 750 CNY, $v\in\mathcal{F}_d$; 550 CNY, $v\in\mathcal{F}_e$\\
					$C_{1,v}$ & 19.5 CNY, $v\in\mathcal{F}_d$; 9.5 CNY, $v\in\mathcal{F}_e$ & $C_2$ & 25 CNY\\
					$y^{\prime}$ & 5 CNY & $\theta$ & 737\\
					$\overline{Q}_v$ & 15,500 kilogram, $v\in\mathcal{F}_d$ & $M_v$ & 31,000 kilogram, $v\in\mathcal{F}_d$\\
					$\xi$ & 1 & $k$ & 0.2 kilojoule/revolution/liter\\
					$N$ & 32 revolution/second & $V$ & 12.54 liter \\
					$\kappa$ & 44 kilojoule/gram & $\eta$ & 0.9\\
					$\eta_{t}$ & 0.4 & $\tau$ & 0\\
					$g$ & 9.81 meter/square second & $\delta$ & 0\\
					$f_a$ & 0.7 & $f_r$ & 0.01\\
					$\varphi$ & 1.2041 kilogram/square meter & $A$ & 8.9 square meter\\
					\hline
				\end{tabular}
			\end{table}
			
			\section{Models considering multiple CW types}\label{Appendix:models}
			\setcounter{table}{0}   
			\setcounter{figure}{0}
			\setcounter{equation}{0}
			\renewcommand{\thetable}{C\arabic{table}}
			\renewcommand{\thefigure}{C\arabic{figure}}
			\renewcommand{\theequation}{C\arabic{equation}}
			
			Here we consider one extension of the model when CW can be categorized as multiple types such as inert waste, non-inert non-hazardous waste, and hazardous waste~\citep{chen2024construction}. Different types of waste have different recycling methods and treatment fees. Specifically, we define $\mathcal{C}$ as the set of all types of CW, and $c\in\mathcal{C}$ is a certain type of CW, and $c = 0$ represents that the trucks run empty. We redefine the upper decision variable $y_{j,v}$ and the lower decision variable $x_{i,j,v,t}$ as $y_{j,v,c}$, $x_{i,j,v,t,c}$, respectively,  $i,j\in\{0\}\cup\mathcal{P}\cup\mathcal{S}\cup\mathcal{D},v\in\mathcal{F}_e\cup\mathcal{F}_d,t\in\mathcal{T}\cup\overline{\mathcal{T}},c\in\mathcal{C}$. We define the models that consider multiple types of CW as model~\textbf{[M4]} and model~\textbf{[M5]}. Models~\textbf{[M4]} and~\textbf{[M5]} are consistent with the solution approach for models ~\textbf{[M2]} and ~\textbf{[M3]}, respectively.
			
			\allowdisplaybreaks[3]
			\begin{align}
				&\text{\bf{[M4]}}\notag\\ &\text{Min }
				\begin{aligned}[t]
					F_{1}^{\prime}(\mathbf{x}^{\prime}|\mathbf{y})=& \sum_{c\in\mathcal{C}}\sum_{i\in\mathcal{P}\cup\mathcal{S}\cup\mathcal{D}}\sum_{j\in\mathcal{P}\cup\mathcal{S}\cup\mathcal{D}}\sum_{v\in\mathcal{F}_d}\sum_{t\in\mathcal{T}}kNV\lambda d_{i,j}x_{i,j,v,t,c}^{\prime}/u \\
					& +\sum_{c\in\mathcal{C}}\sum_{i\in\mathcal{P}\cup\mathcal{S}\cup\mathcal{D}}\sum_{j\in\mathcal{P}\cup\mathcal{S}\cup\mathcal{D}}\sum_{v\in\mathcal{F}_d}\sum_{t\in\mathcal{T}}\overline{Q}_v\gamma\lambda\alpha_{i,j}d_{i,j}x_{i,j,v,t,c}^{\prime} \\
					& +\sum_{c\in\mathcal{C}}\sum_{i\in\mathcal{S}}\sum_{j\in\mathcal{P}\cup\mathcal{D}}\sum_{v\in\mathcal{F}_d}\sum_{t\in\mathcal{T}}Q_v\gamma\lambda\alpha_{i,j}d_{i,j}x_{i,j,v,t,c}^{\prime}\\
					&
					+\sum_{c\in\mathcal{C}}\sum_{i\in\mathcal{P}}\sum_{j\in\mathcal{D}}\sum_{v\in\mathcal{F}_d}\sum_{t\in\mathcal{T}}Q_v\gamma\lambda\alpha_{i,j}d_{i,j}x_{i,j,v,t,c}^{\prime} \\
					& +\sum_{c\in\mathcal{C}}\sum_{i\in\mathcal{P}\cup\mathcal{S}\cup\mathcal{D}}\sum_{j\in\mathcal{P}\cup\mathcal{S}\cup\mathcal{D}}\sum_{v\in\mathcal{F}_d}\sum_{t\in\mathcal{T}}\beta\gamma\lambda d_{i,j}x_{i,j,v,t,c}^{\prime}u^2 \\
					& +\sum_{c\in\mathcal{C}}\sum_{i\in\mathcal{S}}\sum_{j\in\mathcal{P}}\sum_{v\in\mathcal{F}_e\cup\mathcal{F}_d}\sum_{t\in\mathcal{T}}h_jQ_vx_{i,j,v,t,c}^{\prime};
				\end{aligned}\label{M4:obj} \\
				&\text{Subject to:} \notag \\
				&\qquad y_{j,v,c}\geq PL,\forall j\in\mathcal{P},\forall v\in\mathcal{F}_e\cup\mathcal{F}_d,\forall c\in\mathcal{C}; \label{M4:yPL}\\
				&\qquad y_{j,v,c}\leq PU,\forall j\in\mathcal{P},\forall v\in\mathcal{F}_e\cup\mathcal{F}_d,\forall c\in\mathcal{C}; \label{M4:yPU}\\
				\text{where} \notag \\
				&\begin{aligned}
					\mathbf{x}^{\prime}\in \arg\min\{&f^{\prime\prime}(\mathbf{x},\mathbf{y})=\sum_{v\in\mathcal{F}_{e}\cup\mathcal{F}_{d}}\sum_{j\in\mathcal{P}\cup\mathcal{S}\cup D}\sum_{t\in\mathcal{T}}C_{0,v}x_{0,j,v,t,0} \\
					&+\sum_{c\in\mathcal{C}}\sum_{v\in\mathcal{F}_{e}\cup\mathcal{F}_{d}}\sum_{i\in\mathcal{P}\cup\mathcal{S}\cup\mathcal{D}}\sum_{j\in\mathcal{P}\cup\mathcal{S}\cup\mathcal{D}}\sum_{t\in\mathcal{T}}C_{1,v}r_{i,j}x_{i,j,v,t,c} \\
					&-\sum_{c\in\mathcal{C}}\sum_{v\in\mathcal{F}_{e}\cup\mathcal{F}_{d}}\sum_{i\in\mathcal{S}\cup\mathcal{P}}\sum_{j\in\mathcal{D}}\sum_{t\in\mathcal{T}}C_{2}Q_{v}x_{i,j,v,t,c}\\
					&-\sum_{c\in\mathcal{C}}\sum_{v\in\mathcal{F}_{e}\cup\mathcal{F}_{d}}\sum_{i\in\mathcal{S}}\sum_{j\in\mathcal{P}}\sum_{t\in\mathcal{T}}(C_{2}-y_{j,v,c})Q_{v}x_{i,j,v,t,c}\};
				\end{aligned}\label{M4:C4}\\
				&\text{Subject to:} \notag\\
				& \qquad \sum_{j\in\mathcal{P}\cup\mathcal{S}\cup\mathcal{D}}\sum_{t\in\mathcal{T}}x_{0,j,v,t,0}\leq N_v,\forall v\in\mathcal{F}_e\cup\mathcal{F}_d; \label{M4:C5}\\
				& \qquad \sum_{j\in\mathcal{P}\cup\mathcal{S}\cup\mathcal{D}}\sum_{t\in\mathcal{T}}x_{0,j,v,t,0}-\sum_{i\in\mathcal{P}\cup\mathcal{S}\cup\mathcal{D}}\sum_{t\in\mathcal{T}}x_{i,0,v,t,0}=0,\forall v\in\mathcal{F}_e\cup\mathcal{F}_d; \label{M4:C6}\\
				& \qquad \sum_{c\in\mathcal{C}}\sum_{i\in\{0\}\cup \mathcal{P}\cup \mathcal{S}\cup \mathcal{D}}x_{i,j,v,t-r_{i,j}-1,c}-\sum_{c\in\mathcal{C}}\sum_{i\in\{0\}\cup \mathcal{P}\cup \mathcal{S}\cup \mathcal{D}}x_{j,i,v,t,c}=0,\notag\\
				& \qquad \forall j\in\mathcal{P}\cup\mathcal{S}\cup\mathcal{D},\forall v\in\mathcal{F}_e\cup\mathcal{F}_d,\forall t\in\mathcal{T}; \label{M4:C7}\\
				& \qquad \sum_{c\in\mathcal{C}}\sum_{i\in\{0\}\cup\mathcal{P}\cup\mathcal{S}\cup\mathcal{D}}\sum_{v\in\mathcal{F}_e\cup\mathcal{F}_d}x_{i,j,v,t-r_{i,j},c}\leq B_j,\forall j\in\mathcal{P}\cup\mathcal{S}\cup\mathcal{D},\forall t\in\mathcal{T}; \label{M4:C8}\\
				& \qquad \sum_{c\in\mathcal{C}}\sum_{j\in\mathcal{P}\cup\mathcal{D}}\sum_{v\in\mathcal{F}_e\cup\mathcal{F}_d}\sum_{t\in\mathcal{T}}Q_vx_{i,j,v,t,c}/1,000 \geq qs_i,\forall i\in\mathcal{S}; \label{M4:C9}\\
				& \qquad \sum_{c\in\mathcal{C}}\sum_{j\in\mathcal{P}\cup\mathcal{D}}\sum_{v\in\mathcal{F}_e\cup\mathcal{F}_d}\sum_{t\in\mathcal{T}}Q_vx_{i,j,v,t,c}/1,000 < (1+\epsilon_2)qs_i,\forall i\in\mathcal{S};\\
				& \qquad
				\sum_{c\in\mathcal{C}}\sum_{i\in\mathcal{P}\cup\mathcal{S}}\sum_{v\in\mathcal{F}_e\cup\mathcal{F}_d}\sum_{t\in\mathcal{T}}Q_vx_{i,j,v,t,c}/1,000 \geq qd_j,\forall j\in\mathcal{D}; \label{M4:C10}\\
				& \qquad
				\sum_{c\in\mathcal{C}}\sum_{i\in\mathcal{P}\cup\mathcal{S}}\sum_{v\in\mathcal{F}_e\cup\mathcal{F}_d}\sum_{t\in\mathcal{T}}Q_vx_{i,j,v,t,c}/1,000 < (1+\epsilon_2)qd_j,\forall j\in\mathcal{D};\\
				& \qquad x_{i,j,v,t,c}=0,\forall i,j\in\{0\}\cup\mathcal{P}\cup\mathcal{S}\cup\mathcal{D},\forall v\in\mathcal{F}_e\cup\mathcal{F}_d,\forall t\in\overline{\mathcal{T}},\forall c\in\mathcal{C} \label{M4:C11};\\
				& \qquad x_{i,j,v,t,c}=0, \exists\mathcal{H}\in\{\{0\},\mathcal{P},\mathcal{S},\mathcal{D}\},\forall i,j\in \mathcal{H},\forall v\in\mathcal{F}_e\cup\mathcal{F}_d,\forall t\in\mathcal{T},\forall c\in\mathcal{C}; \label{M4:C12}\\
				& \qquad x_{0,j,v,t,c}=0,\forall j\in\{0\}\cup\mathcal{P}\cup S\cup D,\forall v\in\mathcal{F}_e\cup\mathcal{F}_d,\forall t\in \mathcal{T},\forall c\in\mathcal{C}\backslash\{0\}; \label{M4:C13}\\
				& \qquad x_{i,0,v,t,c}=0,\forall i\in\{0\}\cup\mathcal{P}\cup S\cup D,\forall v\in\mathcal{F}_e\cup\mathcal{F}_d,\forall t\in \mathcal{T},\forall c\in\mathcal{C}\backslash\{0\}; \label{M4:C14}\\
				& \qquad \sum_{c\in\mathcal{C}} x_{i,j,v,t,c}\in\{0,1,...,B_j\},\forall i,j\in\{0\}\cup\mathcal{P}\cup S\cup D,\forall v\in\mathcal{F}_e\cup\mathcal{F}_d,\forall t\in \mathcal{T}\cup\overline{\mathcal{T}}.\label{M4:C15}
			\end{align}
			
			\begin{align}
				&\text{\bf{[M5]}}\notag\\ &\text{Min }
				\begin{aligned}[t]
					-\sum_{c\in\mathcal{C}}\sum_{i\in \mathcal{S}}\sum_{j\in\mathcal{P}}\sum_{v\in\mathcal{F}_e\cup\mathcal{F}_d}\sum_{t\in\mathcal{T}}y_{j,v,c}Q_vx_{i,j,v,t,c}^{\prime};
				\end{aligned}\label{M5:obj} \\
				&\text{Subject to:} \notag \\
				&\qquad F_{1}^{\prime}(\mathbf{x}^{\prime}|\mathbf{y})=F_{1}^{\prime}(\mathbf{x}^*_{M4}); \label{M5:C2}\\
				&\qquad \text{Constraints}~\eqref{M4:yPL}-\eqref{M4:yPU}; \notag \\
				\text{where} \notag \\
				&\begin{aligned}
					\mathbf{x}^{\prime}\in \arg\min\{f^{\prime\prime}(\mathbf{x},\mathbf{y})\};
				\end{aligned}\label{M5:LL_obj}\\
				&\text{Subject to:} \notag\\
				&\qquad \text{Constraints~\eqref{M4:C5}-\eqref{M4:C15}.} \notag
			\end{align}
			The solution approach of this model could be barely solved by the same routine as presented in Section~\ref{Sec:Ma}. 		
	\end{document}